\documentclass[final,leqno]{siamltex}

\usepackage{diagbox}
\usepackage{multirow}
\usepackage{amsfonts,amssymb} %%%add
\usepackage{amsmath} %%%add
\usepackage{color} %%%add
\usepackage{graphicx} %%%add
\usepackage{epsfig,mathrsfs} %%%add
\usepackage[bf,SL,BF]{subfigure} %%%add
\usepackage{fancyhdr} %%%add
\usepackage{CJK} %%%add
\usepackage{caption} %%%add
\usepackage{wrapfig} %%%add
\usepackage{cases} %%%add
\usepackage{bm}%%%add
\usepackage{algorithm}
\usepackage{algorithmic}

\newtheorem{remark}{Remark}[section] %%%add
\usepackage{booktabs}
\usepackage{mathtools}

\title{High-order  BDF convolution quadrature  for  subdiffusion models  with a singular source term}

\author{Jiankang Shi \thanks{School of Mathematics and Statistics, Gansu Key Laboratory of Applied Mathematics and Complex Systems,
 Lanzhou University, Lanzhou 730000, P.R. China.
Email address: shijk20@lzu.edu.cn}
 \and  Minghua Chen \thanks{School of Mathematics and Statistics, Gansu Key Laboratory of Applied Mathematics and Complex Systems,
 Lanzhou University, Lanzhou 730000, P.R. China.
Email address: chenmh@lzu.edu.cn}
}

\begin{document}

\maketitle

\begin{abstract}
Anomalous diffusion is often modelled in terms of the subdiffusion equation, which { can involve}  a weakly singular source  term.
For this case,  many predominant time stepping methods, including the correction of high-order  BDF schemes  [{\sc Jin, Li, and Zhou},
SIAM J. Sci. Comput., 39 (2017), A3129--A3152], may suffer from a severe order reduction.
 To fill in this gap,
we propose a smoothing method for time stepping schemes, where the singular term is regularized by using a $m$-fold integral-differential calculus
and the equation is discretized by the $k$-step BDF convolution quadrature, called ID$m$-BDF$k$ method.
We prove that the desired $k$th-order convergence  can be recovered  even if the source term is a weakly singular and the initial data is not compatible.
Numerical experiments illustrate the theoretical results.
 \end{abstract}


%\begin{AMS}
%%26A33, 26A30, 65M12
%\end{AMS}

\pagestyle{myheadings}
\thispagestyle{plain}
\markboth{J. SHI AND M. CHEN}{ID$M$-BDF$K$ FOR SUBDIFFUSION WITH SINGULAR TERM}

\section{Introduction}\label{Se:intro}
{ In  this paper we study the convolution quadrature generated by $k$-step backward differentiation formulas (BDF$k$)
for solving the subdiffusion model  with a  weakly singular source term, whose prototype equation is,  for $0<\alpha<1$}
\begin{equation} \label{fee}
\begin{aligned}
{ \partial^{\alpha}_t (u(t)-v)} - A u(t)=g(t):=t^{\mu}\circ f(t),~~\mu>-1
%, ~~\mu>
%\begin{cases}
%-1,~ \circ \text{denotes convolution}\\ -1,~ \circ \text{denotes product}
%\end{cases}
\end{aligned}
\end{equation}
with the initial condition  $u(0)=v$. The operator $A$  denotes the Laplacian $\Delta$ on a convex polyhedral domain $\Omega \subset \mathbb{R}^d$  with a homogenous Dirichlet boundary condition,
and $\mathcal{D}(A) = H^{1}_{0}(\Omega) \cap H^{2}(\Omega)$, where $H^{1}_{0}(\Omega)$, $H^{2}(\Omega)$ denote the standard Sobolev spaces.  The symbol $\circ$ can be either the convolution $*$  or the  product, and the  Riemann-Liouville fractional derivative   is defined by \cite[p.\,62]{Podlubny:99}
$$\partial^{\alpha}_t u(t)= \frac{1}{\Gamma(1-\alpha)}\frac{d}{dt} \int^{t}_{0} {(t-\tau)^{-\alpha}u(\tau)} d\tau.$$
{ It makes sense to allow $\partial^{\alpha}_t u(t)$ to be singular at $t=0$ if $u$ is  absolutely continuous.
This leads to the fractional differential equations involving a singular  source term,
 see   \cite[Eq. (20)]{SBMB:03}, \cite[Eq. (7.24)]{Diethelm:2010},  \cite[Eq. (4.22)]{JinB:2021} and \cite[Eq. (1.6)]{McLean:19}. }

{ Problems of the model  \eqref{fee} arise in a variety of physical, chemical  and geophysical applications \cite{Doerries:22,Maryshev:09,SBMB:03,SMB:09}.
As an example, a singular fractal mobile/immobile model for solute transport   \cite{SBMB:03}   has important applications in practice,
which has been applied successfully to geophysical systems such as groundwater aquifers, rivers and porous media \cite{Doerries:22,SMB:09}.}

%Note that  the subdiffusion  models naturally   imply nonsmooth or low regularity source term  \cite{JinB:2021,Joulin:1985,McLean:19,Audounet:1998,Diethelm:2010,SMB:09,SBMB:03,SZS:22}.
%In this case,   the predominant time-stepping methods, including the correction of high-order  BDF schemes \cite{JLZ:2017,SC:2020}, are likely to exhibit a severe order reduction
%{ for subdiffusion model \eqref{fee}. The aim of this paper is to fill in this gap}.
%
%%This is the motivation for our study  the   subdiffusion model \eqref{fee} with a  weakly singular source term.

Numerical methods for the time discretization of  \eqref{fee} have been investigated  by various  authors.
Nowadays, there exist  two predominant time-stepping schemes  to restore the high-order convergence rate for { the}  model \eqref{fee}.
The first type is that  the  { nonuniform meshes (e.g., graded meshes, geometric meshes)} are employed to capture  the weak singularities at $t=0$  under  the appropriate regularity  of the solution, see \cite{CJB:2021,Kopteva:2021,Liao:2018,Mustapha:15,Mustapha:20,SOG:2017}.
The second type is convolution quadrature.  For example, { an} usual approach for the source term $g(t)$ in  \eqref{fee} is { to write }
%It is well-known that the smoothness of all the data of \eqref{fee} (e.g., $g=0$) do not imply the smoothness of the solution $u$ which has an initial layer  at $t\rightarrow 0^{+}$ (i.e., unbounded near $t=0$)
%\cite{Podlubny:99,SY:11,SOG:2017}.
%There are already two predominant  discretization techniques in time direction to restore the desired convergence rate for subdiffusion  under appropriate regularity source function.
%The first type is that the nonuniform time meshes/graded meshes are employed to compensate/capture  the singularity of
%the continuous solution near $t=0$ under  the appropriate regularity  source function and initial data, see \cite{CJB:2021,Kopteva:2021,Liao:2018,McLeanMustapha:2015,MustaphaAbdallahFurati:2014,Mustapha:20,SOG:2017}.
%See also spectral method with specially designed basis functions \cite{Chen:2020, HouXu:2017,ZayernouriKarniadakis:2013}.
%The second  type is that,  based on correction of high-order  BDF$k$  or   $L_k$ approximation, the desired  high-order convergence rates can be restored even for nonsmooth initial data.
%For fractional ODEs, one idea is to use starting quadrature weights to correct the fractional integrals  \cite{Lubich:86} (or fractional substantial calculus  \cite{CD:15})
% $$J^\alpha g(t)=\frac{1}{\Gamma(\alpha)}\int_0^t(t-\tau)^{\alpha-1}g(\tau)d\tau~~{\rm with}~~g(t)=t^\mu f(t),~~ \mu>-1,$$
%where the algorithms  rely  on expanding the solution into power series of $t$.
%For fractional PDEs, a common practice is to split the source term into
$$
g(t)=g(0)+\sum^{k-1}_{l=1}\frac{t^{l}}{l!}\partial^{l}_{t}g(0)+\frac{t^{k-1}}{(k-1)!} \ast \partial^{k}_{t}g{  (t) }.
$$
Then approximating $g(0)=\partial_{t}J^1g(0)$ by $\partial_{\tau}J^1g(0)$ may yield a modified BDF2 method  with correction in the first step \cite{Cuesta:06}.
Further more, the correction of  high-order  BDF$k$  or  $L_k$ methods  are well developed in   \cite{JLZ:2017,SC:2020,SCYC:2022,YKF:2018} under   the mild  regularity   of the source function $g$. For the low regularity  source term   $g(t)=t^{\mu}$, $\mu>0$,
the correction of high-order BDF$k$ scheme{ s} converge  with the order $\mathcal{O}(\tau^{1+\mu})$, see Lemma 3.2 in \cite{WZ:2020}, which   may lose the high-order accuracy
and exhibit a severe order reduction.
For the weakly singular  source function $g(t)=t^{\mu}$, $\mu>-1$, a  second-order  time-stepping method {   is} provided in \cite{ZT:2022}
by performing  the integral-differential calculus on both sides of  \eqref{fee}.
For the general  { function}  $g(t)=t^{\mu}f(t)$, $\mu>-1$, the second-order schemes are  well developed in \cite{CSZ:2022}
just by performing  the integral-differential { operator}   for the source { term $g$}.
However, it may not offer an important insight into the causes of high-order  BDF convolution quadrature  for { the}  subdiffusion model \eqref{fee}.
For example, an optimal error estimate  of the Newton-Cotes \cite{CQSW:2021,Suli:03} rule $\mathcal{O}\left(\tau^{\min\{m+1,k\}}\right)$ for odd $m$ and $\mathcal{O}\left(\tau^{\min\{m+2,k\}}\right)$ for even $m$, $1\leq m\leq k\leq 6$,
are  difficult to  illustrate, see  Theorem \ref{addtheorema3.1}.

How to design/restore  the desired $k$th-order convergence rate of the $k$-step BDF ($k\leq 6$) convolution quadrature with a weakly singular  source term
for { the} model  \eqref{fee},  is still unavailable in the literature. To fill in this gap,
we propose { and analyze}  a smoothing method for { the} time stepping schemes, where the singular  term is regularized by using a $m$-fold integral-differential calculus
and the equation is discretized by the $k$-step BDF convolution quadrature, called ID$m$-BDF$k$ method or smoothing method.
%Our findings thus offer important insight into simulations and applications of subdiffusion models.
We prove that the desired $k$th-order convergence  can be recovered  even if the source term is a weakly singular and the initial data is not compatible.
Numerical experiments illustrate the theoretical results.

%The paper is organized as follows. In Section \ref{Se:corre}, we introduce the development of the IDk-BDFk scheme for model \eqref{fee}.
%In Section 3 and 4, based on operational calculus,  the detailed convergence analysis of IDk-BDFk are provided, respectively,  for general source function   $f(x,t)$ and certain form  $t^{\mu}f(x)  $.
%Then the desired results with the  low regularity source term  $t^{\mu}\circ f(x,t)$  are obtained in Section 5.
%To show the effectiveness of the presented schemes, the results of numerical experiments are reported in Section 6.
%Finally, we conclude the paper with some remarks in the last section.

\section{ID$m$-BDF$k$ Method (Smoothing Method)}\label{Se:corre}
%We first provide a simple    ID$m$-BDF$k$ method with $0\leq m \leq k \leq 6$  for solving subdiffusion model   \eqref{fee}.
Let $V(t)=u(t)-v$ with $V(0) = 0$. Then we can rewrite  \eqref{fee}     as
\begin{equation}\label{rfee}
 \partial^{\alpha}_{t} V(t) - A V(t)= Av + g(t), \quad 0<t\leq T.
\end{equation}
It is well-known that the  operator $A$ satisfies the  resolvent estimate  \cite{LST:1996,Thomee:2006}
\begin{equation*}\label{resolvent estimate}
\left\| (z-A)^{-1} \right\| \leq c |z|^{-1} \quad \forall z\in \Sigma_{{ \theta}}
\end{equation*}
for all ${ \theta}\in (\pi/2, \pi)$. Here $\Sigma_{\theta}:=\{ z\in \mathbb{C}\backslash \{0\}:|\arg z| < \theta \}$ is a sector of the complex plane $\mathbb{C}$.
It means that  $z^{\alpha}\in \Sigma_{\theta '}$, $\theta '=\alpha\theta<\theta<\pi$ for all $z\in \Sigma_{\theta}$ and
\begin{equation}\label{fractional resolvent estimate}
\big\| \left(z^{\alpha}-A\right)^{-1} \big\| \leq c |z|^{-\alpha} \quad \forall z\in \Sigma_{\theta}.
\end{equation}
{ Here and below $\left\|\cdot \right\|$ and $\left\|\cdot \right\|_{L^2(\Omega)}$  denote the operator norm  \cite[p.\,91]{Thomee:2006}
and usual  norm  \cite[p.\,2]{Thomee:2006} in the space $L^2(\Omega)$, respectively.}

\subsection{Discretization schemes}
Let the  $m$-fold integral calculus \cite[p.\,64]{Podlubny:99}
\begin{equation}\label{smoothing1}
G(t)=J^{m}g(t)=\frac{1}{\Gamma(m)}\int_0^t(t-\tau)^{m-1}g(\tau)d\tau=\frac{t^{m-1}}{\Gamma(m)} \ast g(t),~~1\leq m\leq k\leq 6.
\end{equation}
Note that  $G(t)$ is a { smooth} function and satisfies $G(0)=J^m g(t)|_{t=0}=0$. { Here $J$}  may map a singular point { of $g$}  to a zero point { of $G$}.
%Under certain reasonable assumptions (e.g., $g(t)$ is only continuous for $t\geq0$), we have
%$$J^0 g(t)=\lim_{m\rightarrow 0}J^m g(t)=g(t).$$
The model \eqref{rfee} { then}  becomes
\begin{equation}\label{rrfee}
 \partial^{\alpha}_{t} V(t) - A V(t)=\partial^{m}_{t} \left( \frac{t^{m}}{m!} Av + G(t)\right), \quad 0<t\leq T.
\end{equation}

Let $N \in \mathbb{N}$, $\tau=\frac{T}{N}$ be the uniform time step and $t_{n}=n \tau$,   $n=0,1,\cdots, N$ be a uniform partition of the interval $[0,T]$. Denote   $u^{n}$ as the approximated value  of $u(t)$
{  at $t=t_n$} and $g^{n}=g(t_{n})$.
The convolution quadrature generated by BDF$k$  approximates the  Riemann-Liouville fractional { derivative}   $\partial^{\alpha}_{t}\varphi(t_n)$ by
\begin{equation}\label{2.1}
\partial^{\alpha}_{\tau, k}\varphi^{n}=\frac{1}{\tau^{\alpha}}\sum^{n}_{j=0} \omega_{j}^{(\alpha, k)}\varphi^{n-j}, ~~1\leq k\leq 6.
\end{equation}
Here the weights $ \omega_{j}^{(\alpha, k)}$ are the coefficients in the series expansion
\begin{equation}\label{2.2}
\delta^{\alpha}_{\tau, k}(\xi)=\frac{1}{\tau^{\alpha}}\sum^{\infty}_{j=0} \omega_{j}^{(\alpha, k)}\xi^{j} \quad {\rm with} \quad \delta_{\tau, k}(\xi)=\frac{1}{\tau} \sum_{j=1}^{k}  \frac{1}{j} \left( 1-\xi \right)^{j},
\end{equation}
which can be computed  by the fast Fourier transform \cite[Chapter 7]{Podlubny:99} or recursion \cite{CYZ:21}.

Then ID$m$-BDF$k$ method for \eqref{rrfee} is designed by
%\begin{equation}\label{s2.3}
% {\rm BDF2~ Method:}~~~~  \partial^{\alpha}_{\tau} V^{n} - AV^{n}=  Av + g^{n}.
%\end{equation}
\begin{equation}\label{2.3}
   \partial^{\alpha}_{\tau, k} V^{n} - AV^{n}= \partial^{m}_{\tau, k}  \left( \frac{t^{m}_{n}}{m!} Av + G^{n} \right),~~1\leq m\leq k\leq 6.
\end{equation}
\begin{remark}
For the time semi-discrete scheme{ s} \eqref{2.3}, we require  $v\in \mathcal{D}(A)$.
However one can use the schemes  \eqref{2.3} to prove the error estimates with the nonsmooth data $v\in L^2(\Omega)$, see Theorem \ref{Theorem5.8}.
In this work, we mainly focus on the time semi-discrete scheme{ s}   \eqref{2.3}, since the spatial discretization is well understood.
In fact, we can  choose $v_h=R_hv$ if $v\in \mathcal{D}(A)$ and $v_h=P_hv$ if $v\in L^2(\Omega)$, see \cite{SCYC:2022,Thomee:2006,WYY:2020}.
%\begin{equation*}
%   \partial^{\alpha}_{\tau, k} V_h^{n} - A_hV_h^{n}= \partial^{m}_{\tau, k}  \left( \frac{t^{m}_{n}}{m!} A_hv_h + G_h^{n} \right).
%\end{equation*}
\end{remark}

\subsection{Continuous  solution representation for \eqref{rrfee}}
Applying the Laplace transform in  \eqref{rrfee}, it yields
\begin{equation*}
\widehat{V}(z)= (z^{\alpha}-A)^{-1}\left( z^{-1}Av   + z^{m} \widehat{G}(z) \right).
\end{equation*}
By the inverse Laplace transform, we obtain  \cite{JLZ:2017}
\begin{equation}\label{LT}
V(t) =  \frac{1}{2\pi i} \int_{\Gamma_{\theta, \kappa}} e^{zt}  (z^{\alpha}-A)^{-1}\left( z^{-1}Av   + z^{m} \widehat{G}(z) \right) dz
\end{equation}
with
\begin{equation}\label{Gamma}
\Gamma_{\theta, \kappa}=\{ z\in \mathbb{C}: |z|=\kappa, |\arg z|\leq \theta \} \cup \{ z\in \mathbb{C}: z=re^{\pm i\theta}, r\geq \kappa \}
\end{equation}
and  $\theta \in (\pi/2, \pi)$, $\kappa>0$.

\subsection{Discrete solution representation  for \eqref{2.3}}
%We next develop  the discrete solution representation  for IDk-BDF2 methods.
Given a sequence $(\kappa^n)_0^\infty$  { we denote by}
$$\widetilde{\kappa}(\zeta)=\sum_{n=0}^{\infty}\kappa^n \zeta^n$$
its generating power series.  { Then we have the following result}.
\begin{lemma}\label{Lemma2.1}
Let $\delta_{\tau, k}$ be  given in  \eqref{2.2} and  $G(t)=J^{m}g(t)$, $1\leq m\leq k\leq 6$ in \eqref{smoothing1}.
Then the discrete solution of \eqref{2.3} is  represented by
% $f\in C([0,T];L^{2}(\Omega))$ and $\int^{t}_{0}{(t-s)^{\alpha-1}}||\partial_{s}f(s)||_{L^{2}(\Omega)}ds<\infty$.  Then
%\begin{equation*}
%\begin{split}
%V^{n}
%&=\frac{1}{2\pi i}\int_{\Gamma^{\tau}_{\theta,\kappa}} e^{zt} K(z_{\tau}) \bar{\mu}(e^{-z\tau})(Av+f(0))dz+\frac{1}{2\pi i}\int_{\Gamma^{\tau}_{\theta,\kappa}}e^{zt}z_{\tau}K(z_{\tau})\tau\widetilde{R_{k}}(e^{-z\tau}) dz\\
%&\quad+\frac{1}{2\pi i}\int_{\Gamma^{\tau}_{\theta,\kappa}}e^{zt}z_{\tau}K(z_{\tau})\sum^{k-1}_{l=1}\left(\frac{\gamma_{l}(e^{-z\tau})}{l!}+\sum_{j=1}^{k-1}d^{(k)}_{l,j}e^{-zt_{j}}\right)\tau^{l+1}
%\partial^{l}_{t}f(0)dz
%\end{split}
%\end{equation*}
\begin{equation*}
V^{n}=\frac{\tau}{2\pi i}\int_{\Gamma^{\tau}_{\theta,\kappa}} e^{zt_n} \left( \delta^{\alpha}_{\tau, k}(e^{-z\tau}) -A\right)^{-1} \delta^{m}_{\tau, k}(e^{-z\tau}) \left( \frac{\gamma_{m}(\xi)}{m!} \tau^{m} Av + \widetilde{G} (e^{-z\tau}) \right)dz
\end{equation*}
with $\Gamma^{\tau}_{\theta, \kappa}=\{z\in \Gamma_{\theta, \kappa}: |\Im z|\leq \pi / \tau\}$ and
$ \gamma_{m}(\xi)=\sum^{\infty}_{n=1}n^{m} \xi^{n}=\left( \xi\frac{d}{d\xi} \right)^{m} \frac{1}{1-\xi}$.
% $\gamma_{m}(\xi)=\sum^{\infty}_{n=1}n^{m} \xi^{n},~m\geq 0$. In particular,
%\begin{equation}\label{gamma_l}
%\gamma_{0}(\xi) := \sum^{\infty}_{n=1}  \xi^{n}= \frac{\xi }{1-\xi}, \quad  \gamma_{m}(\xi)=\sum^{\infty}_{n=1}n^{m} \xi^{n}=\left( \xi\frac{d}{d\xi} \right)^{m} \frac{1}{1-\xi}, \quad m \geq 1.
%\end{equation}
\end{lemma}
\begin{proof}
Multiplying  \eqref{2.3} by $\xi^{n}$ and summing over $n$, we obtain
\begin{equation*}
\begin{split}
&\sum^{\infty}_{n=1} \partial^{\alpha}_{\tau, k} V^{n} \xi^{n} - \sum^{\infty}_{n=1}  AV^{n}  \xi^{n} = \sum^{\infty}_{n=1} \partial^{m}_{\tau, k} \left( \frac{t^{m}_{n}}{m!} Av + G^{n}\right) \xi^{n}.
\end{split}
\end{equation*}
From    \eqref{2.1}, \eqref{2.2} and   $V^0=0$, there exists
\begin{equation*}
\begin{split}
\sum^{\infty}_{n=1} \partial^{\alpha}_{\tau, k} V^{n} \xi^{n}
=& \sum^{\infty}_{n=1} \frac{1}{\tau^{\alpha}}\sum^{n}_{j=0}  \omega_{j}^{(\alpha, k)} V^{n-j} \xi^{n}
%=  \sum^{\infty}_{n=0} \frac{1}{\tau^{\alpha}}\sum^{n}_{j=0}  \omega_{j}^{(\alpha, k)} V^{n-j} \xi^{n}\\
%=&  \sum^{\infty}_{j=0} \frac{1}{\tau^{\alpha}}\sum^{\infty}_{n=j}  \omega_{j}^{(\alpha, k)} V^{n-j} \xi^{n}
%= \sum^{\infty}_{j=0} \frac{1}{\tau^{\alpha}}\sum^{\infty}_{n=0} \omega_{k, j} V^{n} \xi^{n+j}
{ =  \frac{1}{\tau^{\alpha}} \sum^{\infty}_{j=0}   \omega_{j}^{(\alpha, k)}  \xi^{j} \sum^{\infty}_{n=0}  V^{n} \xi^{n}}
=  \delta^{\alpha}_{\tau, k} (\xi) \widetilde{V} (\xi).
\end{split}
\end{equation*}
Similarly,  by the identities $\gamma_{m}(\xi)=\sum^{\infty}_{n=1}n^{m} \xi^{n}$,  $m\geq 1$ and $G^0=G(0)=0$, we get
\begin{equation*}
\sum^{\infty}_{n=1} \partial^{m}_{\tau, k} t^{m}_{n} Av  \xi^{n} = \delta^{m}_{\tau, k}(\xi) \gamma_{m}(\xi) \tau^{m} Av, \quad
\sum^{\infty}_{n=1} \partial^{m}_{\tau, k}  G^{n} \xi^{n} =  \delta^{m}_{\tau, k}(\xi) \widetilde{G} (\xi).
\end{equation*}

According to   the above equations,  it yields
\begin{equation}\label{ads2.17}
\widetilde{V}(\xi) =   \left( \delta^{\alpha}_{\tau, k}(\xi) -A\right)^{-1} \delta^{m}_{\tau, k}(\xi) \left( \frac{\gamma_{m}(\xi)}{m!} \tau^{m} Av + \widetilde{G} (\xi) \right).
\end{equation}
{ From}   Cauchy's integral formula, and the change of variables $\xi=e^{-z\tau}$, and  Cauchy's theorem, it implies  \cite{JLZ:2017}
\begin{equation}\label{DLT}
V^{n}=\frac{\tau}{2\pi i}\int_{\Gamma^{\tau}_{\theta,\kappa}}\!\!\!\! e^{zt_n} \left( \delta^{\alpha}_{\tau, k}(e^{-z\tau}) -A\right)^{-1} \delta^{m}_{\tau, k}(e^{-z\tau}) \!\left[ \frac{\gamma_{m}(\xi)}{m!} \tau^{m} Av + \widetilde{G} (e^{-z\tau}) \right]dz
\end{equation}
with $\Gamma^{\tau}_{\theta, \kappa}=\{z\in \Gamma_{\theta, \kappa}: |\Im z|\leq \pi / \tau\}$.
The proof is completed.
\end{proof}

\section{Convergence analysis: General source function $g(t)$}\label{Se:conver}
Based on the { framework}   of convolution quadrature \cite{CSZ:2022,JLZ:2017,SC:2020}, we first
provide the detailed error { analysis}   for  { the} subdiffusion model  \eqref{rfee} with the general source function $g(t)$.
\subsection{A few technical lemmas}
We give some lemmas that will be used.
{ First, we need a few estimates on $\delta_{\tau, k}(e^{-z\tau})$ in \eqref{2.2}.}
\begin{lemma}\cite{JLZ:2017}\label{Lemma 2.3}
Let $\delta_{\tau, k}(\xi)$ with $k \leq 6$ be given in \eqref{2.2}. Then there  exist the positive constants $c_{1},c_{2}$, $c$ and
$\theta \in (\pi/2, \theta_{\varepsilon})$ with  $\theta_{\varepsilon} \in (\pi/2, \pi),~\forall \varepsilon>0$ such that
\begin{equation*}
\begin{split}
& c_{1}|z|\le |\delta_{\tau, k}(e^{-z\tau})| \le c_{2}|z|, \quad   |\delta_{\tau, k}(e^{-z\tau})-z|\le c \tau^{k}|z|^{k+1}, \\
& |\delta^{\alpha}_{\tau, k}(e^{-z\tau})-z^{\alpha}|\le c \tau^{k}|z|^{k+\alpha},~\delta_{\tau, k}(e^{-z\tau}) \in \Sigma_{\pi/2+\varepsilon} \quad {\forall} z\in \Gamma^{\tau}_{\theta,\kappa},
\end{split}
\end{equation*}
where  $\theta\in (\pi/2,\pi)$ is  sufficiently close to $\pi/2$.
\end{lemma}

{ To provide an optimal error estimate  of the Newton-Cotes  rule for  ID$m$-BDF$k$,
the following lemmas  will play an important role.}

% the optimal error estimate  of the Newton-Cotes
%
%or the general  { function}  $g(t)=t^{\mu}f(t)$, $\mu>-1$, the second-order schemes are  well developed in \cite{CSZ:2022}
%just by performing  the integral-differential { operator}   for the source { term $g$}.
%However, it may not offer an important insight into the causes of high-order  BDF convolution quadrature  for { the}  subdiffusion model \eqref{fee}.
%For example, an optimal error estimate  of the Newton-Cotes \cite{CQSW:2021,Suli:03} rule $\mathcal{O}\left(\tau^{\min\{m+1,k\}}\right)$ for odd $m$ and $\mathcal{O}\left(\tau^{\min\{m+2,k\}}\right)$ for even $m$, $1\leq m\leq k\leq 6$,
%are  difficult to  illustrate, see  Theorem \ref{addtheorema3.1}
%
%We next  consider the singular source term $g(t)=t^{\mu}q$ with $ \mu> -1$ for subdiffusion \eqref{rrfee}.
\begin{lemma}\label{Lemma nn3.3} %\cite{SC:2020}
Let $\gamma_{l}(\xi)=\sum^{\infty}_{n=1}n^{l} \xi^{n}$
with $l=0, 1, 2, \ldots, 2k$, $k\leq 6$. Then there { exists} a  positive constant $c$ such that
\begin{equation*}
\left| \frac{\gamma_{l}(e^{-z\tau})}{l!}  \tau^{l+1} - \frac{1}{z^{l+1}} \right| \leq
\left\{ \begin{array}{l@{\quad} l}
c \tau^{l+1},  \qquad &l=0~~{\rm or}~~ l=1,3,\cdots,2k-1,\\
c \tau^{l+2}|z| , \quad& l=2,4,\cdots,2k.
\end{array}\right.
\end{equation*}
\end{lemma}
\begin{proof}
%Taking $\xi=e^{-z\tau}$, we get
%\begin{equation*}
%\gamma_{l}(\xi)=\sum^{\infty}_{n=1}n^{l} \xi^{n}=\left( \xi\frac{d}{d\xi} \right)^{l} \frac{1}{1-\xi} =\frac{\sum_{j=0}^{l} a_{l,j} \xi^{j} } {(1-\xi)^{l+1} }
%\end{equation*}
%with  $a_{0,0}=1$, $a_{0,1}=0$ and
%$$a_{l,j} = j a_{l-1,j} + (l+1-j) a_{l-1,j-1},~ a_{l,0} = a_{l,l+1} = 0,~ l\geq 1.$$
Taking $\xi=e^{-z\tau}$, we get
\begin{equation*}
\gamma_{l}(\xi)=\sum^{\infty}_{n=1}n^{l} \xi^{n} = \frac{\sum_{j=0}^{l} a_{l,j} \xi^{l+1-j} } {(1-\xi)^{l+1} },~~l\geq 0
\end{equation*}
with  $a_{0,0}=1$, $a_{0,1}=0$ and
\begin{equation*}
a_{l,j} = j a_{l-1,j} + (l+1-j) a_{l-1,j-1},~ a_{l,0} = a_{l,l+1} = 0,~ l\geq 1.
\end{equation*}
 In particular, we have  $a_{l,j}=a_{l,l+1-j}$, $l\geq 1$ and
\begin{equation*}
\gamma_{l}(\xi)= \frac{\sum_{j=1}^{l} a_{l,l+1-j} \xi^{l+1-j} } {(1-\xi)^{l+1} }
= \frac{\sum_{j=1}^{l} a_{l,j} \xi^{j} } {(1-\xi)^{l+1} }.
\end{equation*}

By the simple calculation, it yields
\begin{equation*}
\left| \frac{\gamma_{l}(e^{-\eta})}{l!} \eta^{l+1} - 1 \right| \leq \left\{ \begin{array}
 {l@{\quad} l}
  c |\eta|^{l+1},   ~~&l=0~~{\rm or}~~ l=1,3,\cdots,2k-1,\\
  c |\eta|^{l+2} ,  ~~&l=2,4,\cdots,2k,
 \end{array}
 \right.
\end{equation*}
since
\begin{equation*}
\begin{split}
  \frac{\gamma_{l}(e^{-\eta})}{l!} \eta^{l+1} { =\frac{ e^{-\eta}}{\left(1-e^{-\eta}\right)\eta^{-1}}}=
  \frac{1-\frac{\eta}{1!}+\frac{\eta^2}{2!}-\frac{\eta^3}{3!}+\cdots}{ 1-\frac{\eta}{2!}+\frac{\eta^2}{3!}-\frac{\eta^3}{4!}+\cdots},~~~~l=0,
\end{split}
\end{equation*}
and
$$  \frac{\gamma_{l}(e^{-\eta})}{l!} \eta^{l+1}
=\frac{\frac{1}{l!} \sum_{j=1}^la_{l,j}e^{\left( \frac{l+1}{2} -j\right)\eta}}{\left(e^{\frac{\eta}{2}}-e^{-\frac{\eta}{2}}\right)^{l+1}\eta^{-(l+1)}}
=\frac{1+c_2\eta^2+c_4\eta^4+c_6\eta^6+\cdots}{1+d_2\eta^2+d_4\eta^4+d_6\eta^6+\cdots} \quad \forall l\geq 1 $$
with  $c_{2i}=d_{2i}$, $2i\leq l$.
Here the coefficients $c_{2i}$ and $d_{2i}$, $i=1,2,\ldots$,  are computed   by
$$c_{2i}=\frac{1}{l!}\sum_{j=1}^la_{l,j}\frac{1}{(2i)!}  \left( \frac{l+1}{2}-j\right)^{2i},$$
and
$$d_{2i}=\frac{1}{(2i+l+1)!}\sum_{j=0}^{l+1}(-1)^j\binom{l+1}{j}\left( \frac{l+1}{2}-j\right)^{2i+l+1}.$$
The proof is completed.
\end{proof}

%%%=\left( \xi\frac{d}{d\xi} \right)^{l} \frac{1}{1-\xi}
\begin{lemma}\label{Lemma 3.4}
Let $\delta_{\tau, k}(\xi)$ be given in \eqref{2.2} and $\gamma_{l}(\xi)=\sum^{\infty}_{n=1}n^{l} \xi^{n}$
with $l=0, 1, \ldots, k+m$, $1\leq m\leq k\leq 6$. Then there { exists} a  positive constant $c$ such that
\begin{equation*}
\left|\delta^{m}_{\tau, k}(e^{-z\tau}) \frac{\gamma_{l}(e^{-z\tau})}{l!} \tau^{l+1} -\frac{z^m}{z^{l+1}}  \right|\leq \left\{ \begin{array}
 {l@{\quad} l}
  c \tau^{l+1} \left| z \right|^{m} + c \tau^{k}|z|^{k+m-l-1}, & l=0,1,3,\cdots,\\ \\
   c \tau^{l+2} \left| z \right|^{m+1} + c \tau^{k}|z|^{k+m-l-1},&l=2,4,\cdots.
 \end{array}
 \right.
\end{equation*}
%where $\theta\in (\pi/2,\pi)$ is  sufficiently close to $\pi/2$.
\end{lemma}
\begin{proof}
Let
\begin{equation*}
\delta^{m}_{\tau, k}(e^{-z\tau}) \frac{\gamma_{l}(e^{-z\tau})}{l!} \tau^{l+1} -\frac{z^m}{z^{l+1}}= J_{1} + J_{2}
\end{equation*}
with
\begin{equation*}
J_{1}= \delta^{m}_{\tau, k}\left(e^{-z\tau}\right) \left( \frac{\gamma_{l}(e^{-z\tau})}{l!}  \tau^{l+1} - \frac{1}{z^{l+1}} \right) \quad {\rm and} \quad
J_{2}=  \frac{\delta^{m}_{\tau, k}(e^{-z\tau})-z^m}{z^{l+1}}.
\end{equation*}
From  Lemmas \ref{Lemma 2.3} and \ref{Lemma nn3.3},  it leads to
\begin{equation*}
\left|J_{1} \right|\leq \left\{ \begin{array}
 {l@{\quad} l}
  c \tau^{l+1} \left| z \right|^{m} , & l=0~~{\rm or}~~ l=1,3,\cdots,2k-1,\\
   c \tau^{l+2} \left| z \right|^{m+1},& l=2,4,\cdots,2k;
 \end{array}
 \right.
\end{equation*}
and
\begin{equation*}
\left|J_{2} \right|    \leq  c \tau^{k}|z|^{k+1} |z|^{m-1}|z|^{-l-1}  \leq  c \tau^{k}|z|^{k+m-l-1}.
\end{equation*}
By the triangle inequality,  the desired result is obtained.
\end{proof}

{ From  Lemmas \ref{Lemma 2.3}-\ref{Lemma 3.4}, we have the following results, which will be used in the global convergence  analysis.}
\begin{lemma}\label{Lemma 3.11}
Let $\delta^{\alpha}_{\tau, k}$  be given in \eqref{2.2}  and $\gamma_{l}(\xi)=\sum^{\infty}_{n=1}n^{l} \xi^{n}$,
 $l=0,1, \ldots, k+m$, $1\leq m\leq k\leq 6$. Then there { exists} a  positive constant $c$ such that
 \begin{equation*}
\left\|K(z)  \right\|\leq \left\{ \begin{array}
 {l@{\quad} l}
  c \tau^{l+1} \left| z \right|^{m-\alpha} + c \tau^{k}|z|^{k+m-l-1-\alpha}, & l=0~~{\rm or}~~ l=1,3,\cdots,2k-1,\\
   c \tau^{l+2} \left| z \right|^{m+1-\alpha} + c \tau^{k}|z|^{k+m-l-1-\alpha},&l=2,4,\cdots,2k
 \end{array}
 \right.
\end{equation*}
with
$$K(z)=
\left( \delta^{\alpha}_{\tau, k}(e^{-z\tau}) -A\right)^{-1} \delta^{m}_{\tau, k}(e^{-z\tau}) \frac{\gamma_{l}(e^{-z\tau})}{l!} \tau^{l+1} - (z^{\alpha}-A)^{-1} \frac{z^m}{z^{l+1}}. $$
\end{lemma}
\begin{proof}
Let
$
K(z) = I + II
$
with
\begin{equation*}
\begin{split}
I  & = \left( \delta^{\alpha}_{\tau, k}(e^{-z\tau}) -A\right)^{-1} \left[ \delta^{m}_{\tau, k}(e^{-z\tau})  \frac{\gamma_{l}(e^{-z\tau})}{l!}  \tau^{l+1} -  \frac{z^m}{z^{l+1}} \right],  \\
II & = \left[ \left( \delta^{\alpha}_{\tau, k}(e^{-z\tau}) -A\right)^{-1} - (z^{\alpha}-A)^{-1} \right]\frac{z^m}{z^{l+1}}.
\end{split}
\end{equation*}
The resolvent estimate \eqref{fractional resolvent estimate},   Lemmas  \ref{Lemma 2.3} and \ref{Lemma 3.4} imply directly
\begin{equation}\label{discrete fractional resolvent estimate}
\| \left(\delta^{\alpha}_{\tau, k}(e^{-z\tau}) -A\right)^{-1} \| \leq c |z|^{-\alpha},
\end{equation}
and
 \begin{equation*}
\left\|I \right\|\leq \left\{ \begin{array}
 {l@{\quad} l}
  c \tau^{l+1} \left| z \right|^{m-\alpha} + c \tau^{k}|z|^{k+m-l-1-\alpha}, & l=0~~{\rm or}~~ l=1,3,\cdots,2k-1,\\
   c \tau^{l+2} \left| z \right|^{m+1-\alpha} + c \tau^{k}|z|^{k+m-l-1-\alpha},&l=2,4,\cdots,2k.
 \end{array}
 \right.
\end{equation*}
%From \eqref{discrete fractional resolvent estimate} and Lemma \ref{Lemma 3.4}, we obtain
%\begin{equation*}
%\|I\|  \leq c \tau^{l+1}  \left| z \right|^{m-\alpha} + c \tau^{k}|z|^{k+m-1-l-\alpha}
%\end{equation*}
%{\color{red}
%\begin{equation*}
%\|I\|  \leq c \tau^{l+2}  \left| z \right|^{m+1-\alpha} + c \tau^{k}|z|^{k+m-1-l-\alpha}.
%\end{equation*}
%}

According to     Lemma \ref{Lemma 2.3}, \eqref{discrete fractional resolvent estimate} and the identity
\begin{equation}\label{identity1}
\begin{split}
  &\left( \delta^{\alpha}_{\tau, k}(e^{-z\tau}) -A\right)^{-1} - (z^{\alpha}-A)^{-1}\\
 &\quad =\left( z^{\alpha} - \delta^{\alpha}_{\tau, k}(e^{-z\tau}) \right) \left( \delta^{\alpha}_{\tau, k}(e^{-z\tau}) -A\right)^{-1} (z^{\alpha}-A)^{-1},
\end{split}
\end{equation}
we estimate $II$ as following
\begin{equation*}
\|II\|  \leq c \tau^{k} |z|^{k + \alpha}  c |z|^{-\alpha} c |z|^{-\alpha}  |z|^{-l+m-1} \leq c \tau^{k} |z|^{k+m-l-1-\alpha}.
\end{equation*}
Then  the desired result is obtained.
\end{proof}

\begin{lemma}\label{addLemma 3.6}
Let $\delta^{\alpha}_{\tau, k}$  be given in \eqref{2.2}  and $\gamma_{m}(\xi)=\sum^{\infty}_{n=1}n^{m} \xi^{n}$ with $1\leq m \leq k \leq 6$. Then there { exists} a  positive constant $c$ such that
 \begin{equation*}
\left\|\mathcal{K}(z)  \right\|\leq \left\{ \begin{array}
 {l@{\quad} l}
  c \tau^{m+1} \left| z \right|^{m} + c \tau^{k}|z|^{k-1}, & m=1,3,5,\\
   c \tau^{m+2} \left| z \right|^{m+1} + c \tau^{k}|z|^{k-1},&m=2,4,6
 \end{array}
 \right.
\end{equation*}
with
$$\mathcal{K}(z)=
\left(\delta^{\alpha}_{\tau, k}(e^{-z\tau}) -A\right)^{-1} \delta^{m}_{\tau, k}(e^{-z\tau}) \frac{\gamma_{m}(\xi)}{m!} \tau^{m+1} A - (z^{\alpha}-A)^{-1} z^{-1} A. $$
\end{lemma}

\begin{proof}
Since
$(z^{\alpha}-A)^{-1} z^{-1}  A = - z^{-1}   + (z^{\alpha}-A)^{-1} z^{\alpha} z^{-1}$ and
\begin{equation*}
\begin{split}
& \left(\delta^{\alpha}_{\tau, k}(e^{-z\tau}) -A\right)^{-1} \delta^{m}_{\tau, k}(e^{-z\tau}) \frac{\gamma_{m}(\xi)}{m!} \tau^{m+1} A \\
& =  - \delta^{m}_{\tau, k}(e^{-z\tau})\frac{\gamma_{m}(\xi)}{m!} \tau^{m+1}  + \left( \delta^{\alpha}_{\tau, k}(e^{-z\tau}) -A\right)^{-1} \delta^{\alpha}_{\tau, k}(e^{-z\tau}) \delta^{m}_{\tau, k}(e^{-z\tau}) \frac{\gamma_{m}(\xi)}{m!} \tau^{m+1}.
\end{split}
\end{equation*}
Then we can split $\mathcal{K}(z)$ as
\begin{equation*}
\mathcal{K}(z) = J_{1} + J_{2} + J_{3} + J_{4}
\end{equation*}
with
\begin{equation*}
\begin{split}
J_{1} =& \left( \delta^{\alpha}_{\tau, k}(e^{-z\tau}) -A\right)^{-1} \delta^{\alpha}_{\tau, k}(e^{-z\tau}) \left( \delta^{m}_{\tau, k}(e^{-z\tau})  \frac{\gamma_{m}(e^{-z\tau})}{m!} \tau^{m+1} - z^{-1} \right), \\
J_{2} =& \left( \delta^{\alpha}_{\tau, k}(e^{-z\tau}) -A\right)^{-1} \left( \delta^{\alpha}_{\tau, k}(e^{-z\tau}) -  z^{\alpha} \right) z^{-1}, \\
J_{3} =& \left( \left( \delta^{\alpha}_{\tau, k}(e^{-z\tau}) -A\right)^{-1} - (z^{\alpha}-A)^{-1} \right) z^{\alpha - 1}, \quad
J_{4} = z^{-1} - \delta^{m}_{\tau, k}(e^{-z\tau})\frac{\gamma_{m}(\xi)}{m!} \tau^{m+1}.
\end{split}
\end{equation*}

From  \eqref{discrete fractional resolvent estimate}, \eqref{identity1}  and Lemmas \ref{Lemma 2.3}, \ref{Lemma 3.4}, we  estimate $J_{1}$, $J_{4}$ and $J_2$, $J_3$ as following
\begin{equation*}
\left\|J_1 \right\|\leq c\left\|J_4 \right\|\leq \left\{ \begin{array}
 {l@{\quad} l}
 c  \tau^{m+1} \left| z \right|^{m} + c \tau^{k} \left| z \right|^{k-1}, & m=1,3,5,\\
   c \tau^{m+2} \left| z \right|^{m+1} + c \tau^{k}|z|^{k-1},&m=2,4,6,
 \end{array}
 \right.
\end{equation*}
%and
%\begin{equation*}
%\left\|J_4 \right\|\leq \left\{ \begin{array}
% {l@{\quad} l}
% c  \tau^{m+1} \left| z \right|^{m} + c \tau^{k} \left| z \right|^{k-1}, & m=0,1,3,5,\\
%   c \tau^{m+2} \left| z \right|^{m+1} + c \tau^{k}|z|^{k-1},&m=2,4,6.
% \end{array}
% \right.
%\end{equation*}
%\begin{equation*}
%\left|\delta^{m}_{\tau, k}(e^{-z\tau}) \frac{\gamma_{l}(e^{-z\tau})}{l!} \tau^{l+1} -\frac{z^m}{z^{l+1}}  \right|\leq \left\{ \begin{array}
% {l@{\quad} l}
%  c \tau^{l+1} \left| z \right|^{m} + c \tau^{k}|z|^{k+m-l-1}, & l=0,1,3,\cdots,\\
%   c \tau^{l+2} \left| z \right|^{m+1} + c \tau^{k}|z|^{k+m-l-1},&l=2,4,\cdots.
% \end{array}
% \right.
%\end{equation*}
and
\begin{equation*}
\begin{split}
&\left\| J_{2}  \right\| \leq c |z|^{-\alpha} \tau^{k}|z|^{k+\alpha} |z|^{-1}  \leq   c \tau^{k} \left| z \right|^{k-1},
~~~~\left\| J_{3}  \right\|  \leq   c \tau^{k}|z|^{k-1}.
\end{split}
\end{equation*}
The proof is completed.
\end{proof}

\subsection{Error analysis for general source function $g(t)$}
%we provide the detailed convergence analysis
Let    $G(t)=J^{m}g(t)$, $1\leq m \leq k\leq 6$ be defined by \eqref{smoothing1}.
%$v\in L^{2}(\Omega)$, $g\in C^{k-1}([0,T]; L^{2}(\Omega))$ and $\int_{0}^{t}  (t-s)^{\alpha-1} \left\| g^{(k)}(s) \right\| ds <\infty$.
The Taylor expansion of general  source  function with the remainder term in integral form is given by
\begin{equation}\label{gs3.3}
\begin{split}
\frac{t^{m-1}}{(m-1)!} \ast g(t)=G(t) & =  \sum_{l=0}^{k+m-1}  \frac{t^l}{l!}  G^{(l)}(0)  + \frac{t^{k+m-1}}{(k+m-1)!}  \ast  G^{(k+m)}(t)\\
%& =  \sum_{j=-m}^{k-1}  \frac{t^{(j+m)}}{(j+m)!}  g^{(j)}(0)  + \frac{t^{k+m-1}}{(k+m-1)!}  \ast  g^{(k)}(t).
& =  \sum_{l=0}^{k+m-1}  \frac{t^l}{l!} g^{(l-m)}(0)  + \frac{t^{k+m-1}}{(k+m-1)!}  \ast  g^{(k)}(t)
\end{split}
\end{equation}
with $g^{(-i)}(0)=J^{i}g(0)=0$, $i \geq 1$.
Then  we obtain   the following results.
\begin{lemma}\label{lemma3.9}
Let $V(t_{n})$ and $V^{n}$ be the solutions of \eqref{rrfee} and \eqref{2.3}, respectively.
Let $v=0$ and $G(t):=\frac{t^{l}}{l!} g^{(l-m)}(0)$ with $l=0,1, \ldots, k+m-1$, $1\leq m \leq k\leq 6$.
Then the following error estimate holds for any $t_n>0$
 \begin{equation*}
\left\| V(t_{n}) \!-\! V^{n}  \right\|_{ {L^2(\Omega)}}\!\leq\! \left\{ \begin{split}
 \!\!\left( c\tau^{l+1}t_{n}^{\alpha-m-1}\! +\! c\tau^{k}t_{n}^{\alpha+l-k-m} \right)\! \left\| g^{(l-m)}(0)\right\|_{ {L^2(\Omega)}}, & l=0,1,3,5,\cdots\\ \\
 \!\!  \left( c\tau^{l+2}t_{n}^{\alpha-m-2}\! +\! c\tau^{k}t_{n}^{\alpha+l-k-m} \right)\! \left\| g^{(l-m)}(0)\right\|_{ {L^2(\Omega)}},&l=2,4,6,\cdots.
 \end{split}
 \right.
\end{equation*}
\end{lemma}
\begin{proof}
From    \eqref{LT} and Lemma  \ref{Lemma2.1}, it leads to
\begin{equation*}
V(t_{n})=\frac{1}{2\pi i}\int_{\Gamma_{\theta,\kappa}}{e^{zt_{n}}(z^{\alpha}-A)^{-1} \frac{1}{z^{l+1-m}} g^{(l-m)}(0)}dz,
\end{equation*}
and
\begin{equation*}
V^{n}=\frac{1}{2\pi i}\int_{\Gamma^{\tau}_{\theta,\kappa}} e^{zt_{n}} \left( \delta^{\alpha}_{\tau, k}(e^{-z\tau}) -A\right)^{-1} \delta^{m}_{\tau, k}(e^{-z\tau}) \frac{\gamma_{l}(e^{-z\tau})}{l!} \tau^{l+1} g^{(l-m)}(0) dz
\end{equation*}
with $\gamma_{l}(\xi)=\sum^{\infty}_{n=1}n^{l} \xi^{n}$.
Then we have
\begin{equation*}
V(t_{n})-V^{n}=J_1 + J_2
\end{equation*}
with
%\begin{equation*}
%\begin{split}
%J_1
%& = \frac{1}{2\pi i}\int_{\Gamma^{\tau}_{\theta,\kappa}} e^{zt_{n}}  \frac{ \left(z^{\alpha} - A\right)^{-1} }{z^{l+1-m}}    g^{(l-m)}(0) dz \\
%& \quad  - \frac{1}{2\pi i}\int_{\Gamma^{\tau}_{\theta,\kappa}} e^{zt_{n}}  \left( \delta^{\alpha}_{\tau, k}(e^{-z\tau})  - A \right)^{-1} \delta^{m}_{\tau, k}(e^{-z\tau}) \frac{\gamma_{l}(e^{-z\tau})}{l!} \tau^{l+1}     g^{(l-m)}(0) dz,
%\end{split}
%\end{equation*}
\begin{equation*}
\begin{split}
J_1
& = \frac{1}{2\pi i}\int_{\Gamma^{\tau}_{\theta,\kappa}} e^{zt_{n}} K(z)  g^{(l-m)}(0) dz,~~K(z){\rm ~in~ Lemma~\ref{Lemma 3.11}, }
\end{split}
\end{equation*}
%$$\kappa(z)=
%\left( \delta^{\alpha}_{\tau, k}(e^{-z\tau}) -A\right)^{-1} \delta^{m}_{\tau, k}(e^{-z\tau}) \frac{\gamma_{l}(e^{-z\tau})}{l!} \tau^{l+1} - (z^{\alpha}-A)^{-1} \frac{z^m}{z^{l+1}}. $$
%\ref{Lemma 3.11}
and
\begin{equation*}
\begin{split}
J_2 =\frac{1}{2\pi i}\int_{\Gamma_{\theta,\kappa}\setminus\Gamma^{\tau}_{\theta,\kappa}}{e^{zt_{n}}\left(z^{\alpha} - A\right)^{-1} \frac{1}{z^{l+1-m}} g^{(l-m)}(0)}dz.
\end{split}
\end{equation*}
According to the triangle inequality, \eqref{fractional resolvent estimate} and Lemma \ref{Lemma 3.11}, { it yields}
% \begin{equation*}
%\left\|\kappa(z)  \right\|\leq \left\{ \begin{array}
% {l@{\quad} l}
%  c \tau^{l+1} \left| z \right|^{m-\alpha} + c \tau^{k}|z|^{k+m-l-1-\alpha}, & l=0~~{\rm or}~~ l=1,3,\cdots,2k-1,\\
%   c \tau^{l+2} \left| z \right|^{m+1-\alpha} + c \tau^{k}|z|^{k+m-l-1-\alpha},&l=2,4,\cdots,2k
% \end{array}
% \right.
%\end{equation*}
%with
%$$\kappa(z)=
%\left( \delta^{\alpha}_{\tau, k}(e^{-z\tau}) -A\right)^{-1} \delta^{m}_{\tau, k}(e^{-z\tau}) \frac{\gamma_{l}(e^{-z\tau})}{l!} \tau^{l+1} - (z^{\alpha}-A)^{-1} \frac{z^m}{z^{l+1}}. $$
\begin{equation*}
\begin{split}
  \| J_1 \|_{ {L^2(\Omega)}}
  &  \leq c  \int^{\frac{\pi}{\tau\sin\theta}}_{\kappa} e^{rt_{n}\cos\theta} \left(\tau^{l+1}  r^{m-\alpha} +  \tau^{k} r^{k+m-1-l-\alpha} \right) dr \left\|g^{(l-m)}(0)\right\|_{ {L^2(\Omega)}} \\
  & \quad + c \int^{\theta}_{-\theta}e^{\kappa t_{n} \cos\psi} \left(\tau^{l+1}  \kappa^{m+1-\alpha} +  \tau^{k} \kappa^{k+m-l-\alpha} \right) d\psi  \left\|g^{(l-m)}(0)\right\|_{ {L^2(\Omega)}} \\
  &  \leq  \left( c\tau^{l+1}t_{n}^{\alpha-m-1} + c\tau^{k}t_{n}^{\alpha+l-k-m} \right) \left\| g^{(l-m)}(0)\right\|_{ {L^2(\Omega)}},~ l=0~{\rm or}~ l=1,3,\cdots,
\end{split}
\end{equation*}
and
\begin{equation*}
  \| J_1 \|_{ {L^2(\Omega)}}  \leq  \left( c\tau^{l+2}t_{n}^{\alpha-m-2} + c\tau^{k}t_{n}^{\alpha+l-k-m} \right) \left\| g^{(l-m)}(0)\right\|_{ {L^2(\Omega)}},~~l=2,4,\cdots,
\end{equation*}
where  we use
\begin{equation}\label{ad3.3.09}
\begin{split}
& \int^{\frac{\pi}{\tau\sin\theta}}_{\kappa} e^{rt_{n}\cos\theta} r^{k+m-1-l-\alpha}dr
%= t_n^{\alpha+l-k-m} \int^{\frac{t_n\pi}{\tau\sin\theta}}_{t_n\kappa} e^{s\cos\theta} s^{k+m-1-l- \alpha}ds
\leq c t_n^{\alpha+l-k-m},\\
& \int^{\theta}_{-\theta}e^{\kappa t_{n} \cos\psi} \kappa^{k+m-l-\alpha} d\psi
%= t_n^{\alpha+l-k-m} \int^{\theta}_{-\theta}e^{\kappa t_{n} \cos\psi} \left(\kappa t_{n}\right)^{k+m-l- \alpha} d\psi
\leq c t_n^{\alpha+l-k-m}.
\end{split}
\end{equation}
From   \eqref{fractional resolvent estimate}, one has
\begin{equation*}
\begin{split}
\|J_2 \|_{ {L^2(\Omega)}}
&\leq c  \int^{\infty}_{\frac{\pi}{\tau\sin\theta}} e^{rt_{n}\cos\theta}r^{m-l-1-\alpha}dr \left\| g^{(l-m)}(0) \right\|_{ {L^2(\Omega)}}\\
&\leq c\tau^{k}  \int^{\infty}_{\frac{\pi}{\tau\sin\theta}} e^{rt_{n}\cos\theta}r^{k+m-1-l-\alpha}dr  \left\| g^{(l-m)}(0) \right\|_{ {L^2(\Omega)}}\\
&\leq  c\tau^{k}t_{n}^{\alpha+l-k-m} \left\| g^{(l-m)}(0) \right\|_{ {L^2(\Omega)}}.
\end{split}
\end{equation*}
Here we { use} $1\leq (\frac{\sin \theta}{\pi})^{k} \tau^{k} r^{k}$ with $r\geq \frac{\pi}{\tau \sin \theta}$.
The proof is completed.
\end{proof}

%%%%%%\partial^{2}_{s}
\begin{lemma}\label{lemma3.10}
Let $V(t_{n})$ and $V^{n}$ be the solutions of \eqref{rrfee} and \eqref{2.3}, respectively.
Let $v=0$,  $G(t):=\frac{t^{k+m-1}}{(k+m-1)!}  \ast  g^{(k)}(t)$, $1\leq m \leq k\leq 6$ and  $\int_{0}^{t}  (t-s)^{\alpha-1} \| g^{(k)}(s) \|_{ {L^2(\Omega)}} ds <\infty$. Then the following error estimate holds for any $t_n>0$
\begin{equation*}
\left\|V(t_{n})-V^{n}\right\|_{ {L^2(\Omega)}}\leq c\tau^{k} \int_{0}^{t_{n}} (t_n-s)^{\alpha-1} \left\| g^{(k)}(s) \right\|_{ {L^2(\Omega)}}ds.
\end{equation*}
\end{lemma}
\begin{proof}
From the continuous  solution representation in  \eqref{LT}, we have
%From   \eqref{LT}, we obtain
\begin{equation}\label{nas3.6}
\begin{split}
V(t_{n})
&=(\mathscr{E}(t)\ast G(t))(t_{n})
%&=\left(\mathscr{E}(t)\ast \left(\frac{t^{k+m-1}}{(k+m-1)!}  \ast  g^{(k)}(t) \right)\right)(t_{n}) \\
=\left(\left(\mathscr{E}(t)\ast \frac{t^{k+m-1}}{(k+m-1)!}     \right) \ast  g^{(k)}(t) \right)(t_{n})
\end{split}
\end{equation}
with
\begin{equation}\label{nas3.007}
  \mathscr{E}(t)= \frac{1}{2\pi i} \int_{\Gamma_{\theta,\kappa}} e^{zt}(z^{\alpha} - A)^{-1} z^{m} dz.
\end{equation}

Let  $\sum^{\infty}_{n=0}\mathscr{E}^{n}_{\tau}\xi^{n}=\widetilde{\mathscr{E_{\tau}}}(\xi):=\left( \delta^{\alpha}_{\tau, k}(\xi) -A\right)^{-1} \delta^{m}_{\tau, k}(\xi)$.
Then using   \eqref{ads2.17}, it yields
\begin{equation*}
\begin{split}
\widetilde{V}(\xi)
&=\left( \delta^{\alpha}_{\tau, k}(\xi) -A\right)^{-1} \delta^{m}_{\tau, k}(\xi) \widetilde{G}(\xi) = \widetilde{\mathscr{E_{\tau}}}(\xi)\widetilde{G}(\xi)
=\sum^{\infty}_{n=0}\mathscr{E}^{n}_{\tau}\xi^{n}\sum^{\infty}_{j=0}G^j\xi^{j}\\
&=\sum^{\infty}_{n=0}\sum^{\infty}_{j=0}\mathscr{E}^{n}_{\tau} G^j \xi^{n+j}=\sum^{\infty}_{j=0}\sum^{\infty}_{n=j}\mathscr{E}^{n-j}_{\tau} G^j \xi^{n}
=\sum^{\infty}_{n=0}\sum^{n}_{j=0}\mathscr{E}^{n-j}_{\tau} G^j \xi^{n}=\sum^{\infty}_{n=0}V^n\xi^{n}
\end{split}
\end{equation*}
with
\begin{equation*}
V^{n}=\sum^{n}_{j=0}\mathscr{E}^{n-j}_{\tau} G^j:=\sum^{n}_{j=0}\mathscr{E}^{n-j}_{\tau} G(t_{j}).
\end{equation*}

According to  the Cauchy's integral formula and the change of variables $\xi=e^{-z\tau}$, we get the representation of the $\mathscr{E}^{n}_{\tau}$ as following %for arbitrary $\rho \in (0,1)$
 \begin{equation*}
\mathscr{E}^{n}_{\tau}=\frac{1}{2\pi i}\int_{|\xi|=\rho}{\xi^{-n-1}\widetilde{\mathscr{E_{\tau}}}(\xi)}d\xi
=\frac{\tau}{2\pi i}\int_{\Gamma^{\tau}_{\theta,\kappa}} {e^{zt_n}\left( \delta^{\alpha}_{\tau, k}(e^{-z\tau}) -A\right)^{-1} \delta^{m}_{\tau, k}(e^{-z\tau}) }dz.
\end{equation*}
%where  $\theta\in (\pi/2,\pi)$ is  sufficiently close to $\pi/2$ and $\kappa=t_{n}^{-1}$ in \eqref{Gamma}.

From  \eqref{discrete fractional resolvent estimate}, Lemma \ref{Lemma 2.3} and $\tau t^{-1}_{n} = \frac{1}{n}\leq 1$, it implies
\begin{equation}\label{3.0002}
\|\mathscr{E}^{n}_{\tau}\| \leq c \tau \left( \int^{\frac{\pi}{\tau\sin\theta}}_{\kappa} e^{rt_{n}\cos\theta} r^{m-\alpha}dr +\int^{\theta}_{-\theta}e^{\kappa t_{n}\cos\psi} \kappa^{m+1-\alpha}  d\psi\right)
%\leq c\tau t_{n}^{\alpha-m-1}
\leq c t_{n}^{\alpha-m}.
\end{equation}
Let $ \mathscr{E}_{\tau}(t)=\sum^{\infty}_{n=0}\mathscr{E}^{n}_{\tau}\delta_{t_{n}}(t)$, with $\delta_{t_{n}}$ being the Dirac delta function at $t_{n}$.
Then
\begin{equation}\label{nas3.8}
\begin{split}
(\mathscr{E}_{\tau}(t)\ast G(t))(t_{n})
& = \left(\sum^{\infty}_{j=0}\mathscr{E}^{j}_{\tau}\delta_{t_{j}}(t) \ast G(t) \right)(t_{n})\\
& = \sum^{n}_{j=0}\mathscr{E}^{j}_{\tau} G(t_{n}-t_{j})
  = \sum^{n}_{j=0}\mathscr{E}^{n-j}_{\tau} G(t_{j})=V^{n}.
\end{split}
\end{equation}
Moreover, using the above equation, there { exists}
\begin{equation*}
\begin{split}
  \widetilde{(\mathscr{E}_{\tau}\ast t^{l})}(\xi)
& = \sum^{\infty}_{n=0} \sum^{n}_{j=0}\mathscr{E}^{n-j}_{\tau}t^{l}_{j}\xi^{n}  =\sum^{\infty}_{j=0} \sum^{\infty}_{n=j}\mathscr{E}^{n-j}_{\tau}t^{l}_{j}\xi^{n}
  =\sum^{\infty}_{j=0} \sum^{\infty}_{n=0}\mathscr{E}^{n}_{\tau}t^{l}_{j}\xi^{n+j}\\
& =\sum^{\infty}_{n=0}\mathscr{E}^{n}_{\tau}\xi^{n}\sum^{\infty}_{j=0}t^{l}_{j}\xi^{j}  =\widetilde{\mathscr{E_{\tau}}}(\xi) \tau^{l} \sum^{\infty}_{j=0}j^{l}\xi^{j}
  =\widetilde{\mathscr{E}_{\tau}}(\xi) \tau^{l} \gamma_{l}(\xi).
\end{split}
\end{equation*}

Combining  \eqref{nas3.6}, \eqref{nas3.8} and Lemma \ref{lemma3.9}, we have
 \begin{equation}\label{aadd3.8}
\left\| \left((\mathscr{E}_{\tau}-\mathscr{E}) \ast \frac{t^{l}}{l!} \right)(t_n)  \right\|\leq \left\{ \begin{array}
 {l@{\quad} l}
 c\tau^{l+1}t_{n}^{\alpha-m-1} + c\tau^{k}t_{n}^{\alpha+l-k-m}, & l=0,1,3,\cdots\\ \\
 c\tau^{l+2}t_{n}^{\alpha-m-2} + c\tau^{k}t_{n}^{\alpha+l-k-m},&l=2,4,6,\cdots
 \end{array}
 \right.
\end{equation}
with $l\leq k+m-1.$
%Similarly, we have
%\begin{equation}\label{nm3.11}
%\left\|\left((\mathscr{E}_{\tau}-\mathscr{E}) \ast 1 \right)(t)\right\| \leq c\tau t_{n}^{\alpha+l-3}.
%\end{equation}
%For $l=0$, we also have the above estimate.

Next, we prove the following inequality \eqref{3.0003}  for $t>0$
\begin{equation}\label{3.0003}
\left\|\left((\mathscr{E}_{\tau}-\mathscr{E}) \ast \frac{t^{k+m-1}}{(k+m-1)!} \right)(t)\right\| \leq c\tau^{k}t^{\alpha-1} \quad \forall t\in (t_{n-1},t_{n}).
\end{equation}
Applying the  Taylor series expansion of $\mathscr{E}(t)$ at $t=t_{n}$, we get
\begin{equation*}
\begin{split}
 \left( \mathscr{E} \ast \frac{t^{k+m-1}}{(k+m-1)!} \right)(t)
 &= \sum_{l=0}^{k+m-1}   \frac{(t-t_{n})^{l}}{l!} \left( \mathscr{E} \ast \frac{t^{k+m-l-1}}{(k+m-l-1)!} \right)(t_{n})\\
 &\quad +\frac{1}{(k+m-1)!} \int^{t}_{t_{n}}(t-s)^{k+m-1} \mathscr{E}(s)ds,
\end{split}
\end{equation*}
which  also holds  for the convolution $  \left( \mathscr{E}_{\tau} \ast t^{k+m-1} \right)(t) $.
From \eqref{aadd3.8}, it leads to
 \begin{equation*}
 \begin{split}
\left\|(t-t_{n})^{l} \left((\mathscr{E}_{\tau}-\mathscr{E}) \ast t^{k+m-l-1} \right)(t_n)  \right\|
&\leq c\tau^l \left( \tau^{k+m-l}t_{n}^{\alpha-m-1} + \tau^{k}t_{n}^{\alpha-l-1}  \right)\\
&\leq  c \tau^{k} t^{\alpha-1}~~\forall t\in (t_{n-1},t_{n}).
 \end{split}
    \end{equation*}

According to  \eqref{nas3.007}, \eqref{fractional resolvent estimate} and \eqref{ad3.3.09}, one has
\begin{equation*}
\begin{split}
\| \mathscr{E}(t) \|
\leq c \left( \int^{\infty}_{\kappa}e^{rt\cos\theta}r^{m-\alpha}dr + \int^{\theta}_{-\theta}e^{\kappa t \cos\psi}\kappa^{m+1-\alpha}d\psi \right)
\leq c t^{\alpha-m-1}.
\end{split}
\end{equation*}
Moreover, we get
\begin{equation*}
\left\| \int^{t}_{t_{n}}(t-s)^{k+m-1} \mathscr{E}(s)ds \right\|
\leq c \int^{t_{n}}_{t}(s-t)^{k+m-1} s^{\alpha-m-1}ds
%\leq c \int^{t_{n}}_{t}(s-t)^{k-1} s^{\alpha-1}ds
\leq  c \tau^{k} t^{\alpha-1}.
\end{equation*}
By  the definition of $ \mathscr{E}_{\tau}(t)=\sum^{\infty}_{n=0}\mathscr{E}^{n}_{\tau}\delta_{t_{n}}(t)$ in \eqref{nas3.8} and \eqref{3.0002}, we deduce
\begin{equation*}
\begin{split}
\left\| \int^{t}_{t_{n}}(t-s)^{k+m-1} \mathscr{E}_{\tau}(s)ds \right\|
& \leq  (t_n-t)^{k+m-1}  \| \mathscr{E}^{n}_{\tau} \|  \leq c \tau^{k+m} t_{n}^{\alpha-m-1} \\
& \leq c \tau^{k} t_{n}^{\alpha-1} \leq c \tau^{k} t^{\alpha-1} \quad  \forall \ t \in   (t_{n-1},t_{n}).
\end{split}
\end{equation*}
Using  the above inequalities, it yields  \eqref{3.0003}.
The proof is completed.
\end{proof}

{ For simplicity, we take}
 \begin{equation}\label{initialesti}
\left\|J_v  \right\|_{ {L^2(\Omega)}}= \left\{ \begin{split}
 & c \tau^{m+1} t_{n}^{-m-1} \left\| v \right\|_{ {L^2(\Omega)}} + c\tau^{k} t_{n}^{-k} \left\| v \right\|_{ {L^2(\Omega)}}, ~ m=1,3,5,\\
&  c \tau^{m+2} t_{n}^{-m-2} \left\| v \right\|_{ {L^2(\Omega)}} + c\tau^{k} t_{n}^{-k} \left\| v \right\|_{ {L^2(\Omega)}},~m=2,4,6;
 \end{split}
 \right.
\end{equation}
and
 \begin{equation*}
\left\|J_g \right\|_{ {L^2(\Omega)}}= \left\{ \begin{split}
& c \sum\limits_{l=0}^{k-1}\left( \tau^{l+m+1}t_{n}^{\alpha-m-1} + \tau^{k}t_{n}^{\alpha+l-k} \right) \left\| g^{(l)}(0)\right\|_{ {L^2(\Omega)}},  ~l+m=1,3,5,\cdots\\
&  c \sum\limits_{l=0}^{k-1} \left( \tau^{l+m+2}t_{n}^{\alpha-m-2} + \tau^{k}t_{n}^{\alpha+l-k} \right) \left\| g^{(l)}(0)\right\|_{ {L^2(\Omega)}},~l+m=2,4,6,\cdots.
 \end{split}
 \right.
\end{equation*}
{ Then we have the following result. }
\begin{theorem}\label{addtheorema3.1}
Let $V(t_{n})$ and $V^{n}$ be the solutions of \eqref{rrfee} and \eqref{2.3}, respectively. Let $v\in L^{2}(\Omega)$, $g\in C^{k-1}([0,T]; L^{2}(\Omega))$
and $\int_{0}^{t}  (t-s)^{\alpha-1} \left\| g^{(k)}(s) \right\|_{ {L^2(\Omega)}} ds <\infty$.  Then the following error estimate holds for any $t_n>0$
\begin{equation*}
\begin{split}
\left\|V^{n}-V(t_{n})\right\|_{ {L^2(\Omega)}}
&\leq ||J_v||_{ {L^2(\Omega)}}+||J_g||_{ {L^2(\Omega)}}+ c\tau^{k}  \int_{0}^{t_{n}}  (t_n-s)^{\alpha-1} \left\| g^{(k)}(s) \right\|_{ {L^2(\Omega)}}  ds.
\end{split}
\end{equation*}
\end{theorem}
\begin{proof}
Subtracting \eqref{LT} from \eqref{DLT}, we obtain
%\begin{equation}
%V(t) =  \frac{1}{2\pi i} \int_{\Gamma_{\theta, \kappa}} e^{zt}  (z^{\alpha}-A)^{-1}\left( z^{-1}Av   + z^{m} \widehat{G}(z) \right) dz
%\end{equation}
%
%\begin{equation}
%V^{n}=\frac{\tau}{2\pi i}\int_{\Gamma^{\tau}_{\theta,\kappa}} e^{zt_n} \left( \delta^{\alpha}_{\tau, k}(e^{-z\tau}) -A\right)^{-1} \delta^{m}_{\tau, k}(e^{-z\tau}) \left( \frac{\gamma_{m}(\xi)}{m!} \tau^{m} Av + \widetilde{G} (e^{-z\tau}) \right)dz
%\end{equation}
\begin{equation*}
V^{n}-V(t_{n})=I_{1}-I_{2}+I_{3}
\end{equation*}
with the related initial terms
\begin{equation}\label{ada3.a1}
\begin{split}
I_{1} =& \frac{1}{2\pi i} \int_{\Gamma^{\tau}_{\theta,\kappa}} e^{zt_{n}}  \mathcal{K}(z) v dz,~~\mathcal{K}(z){\rm ~in~ Lemma~\ref{addLemma 3.6}},\\
I_{2} =& \frac{1}{2\pi i} \int_{\Gamma_{\theta,\kappa}\backslash \Gamma^{\tau}_{\theta,\kappa}} e^{zt_{n}}  (z^{\alpha}-A)^{-1} z^{-1}  Av dz,\\
%I_{3} =& \frac{\tau}{2\pi i}\int_{\Gamma^{\tau}_{\theta,\kappa}} e^{zt_n} \left( \delta^{\alpha}_{\tau, k}(e^{-z\tau}) -A\right)^{-1} \delta^{m}_{\tau, k}(e^{-z\tau})   \widetilde{G} (e^{-z\tau}) dz \\
%& - \frac{1}{2\pi i} \int_{\Gamma_{\theta, \kappa}} e^{zt_n}  (z^{\alpha}-A)^{-1} z^{m} \widehat{G}(z) dz.
\end{split}
\end{equation}
and the related source term
\begin{equation*}
\begin{split}
%I_{1} =& \frac{1}{2\pi i} \int_{\Gamma^{\tau}_{\theta,\kappa}} e^{zt_{n}}  \mathcal{K}(z) v dz,~~\mathcal{K}(z){\rm ~in~ Lemma~\ref{addLemma 3.6}},\\
%I_{2} =& \frac{1}{2\pi i} \int_{\Gamma_{\theta,\kappa}\backslash \Gamma^{\tau}_{\theta,\kappa}} e^{zt_{n}}  (z^{\alpha}-A)^{-1} z^{-1}  Av dz,\\
I_{3} =& \frac{\tau}{2\pi i}\int_{\Gamma^{\tau}_{\theta,\kappa}} e^{zt_n} \left( \delta^{\alpha}_{\tau, k}(e^{-z\tau}) -A\right)^{-1} \delta^{m}_{\tau, k}(e^{-z\tau})   \widetilde{G} (e^{-z\tau}) dz \\
& - \frac{1}{2\pi i} \int_{\Gamma_{\theta, \kappa}} e^{zt_n}  (z^{\alpha}-A)^{-1} z^{m} \widehat{G}(z) dz.
\end{split}
\end{equation*}

Similar to the way performed   in Lemma \ref{lemma3.9}, we estimate
 \begin{equation}\label{ad3.I1}
\left\|I_{1}  \right\|_{ {L^2(\Omega)}}\leq \left\{ \begin{array}
 {l@{\quad} l}
  c \tau^{m+1} t_{n}^{-m-1} \left\| v \right\|_{ {L^2(\Omega)}} + c\tau^{k} t_{n}^{-k} \left\| v \right\|_{ {L^2(\Omega)}}, & m=1,3,5,\\
  c \tau^{m+2} t_{n}^{-m-2} \left\| v \right\|_{ {L^2(\Omega)}} + c\tau^{k} t_{n}^{-k} \left\| v \right\|_{ {L^2(\Omega)}},&m=2,4,6.
 \end{array}
 \right.
\end{equation}
Using the resolvent estimate \eqref{fractional resolvent estimate}, we estimate the second term $I_{2}$ as following
\begin{equation} \label{add3.162}
\left\|I_{2}\right\|_{ {L^2(\Omega)}}
\leq c\int_{\Gamma_{\theta,\kappa}\backslash \Gamma^{\tau}_{\theta,\kappa}} {\left|e^{zt_{n}}\right||z|^{-1} \left\| v \right\|_{ {L^2(\Omega)}}} |dz|
\leq c\tau^{k} t^{-k}_{n}\left\| v \right\|_{ {L^2(\Omega)}},
\end{equation}
since
\begin{equation}\label{add3.16}
\begin{split}
\int_{\Gamma_{\theta,\kappa}\backslash \Gamma^{\tau}_{\theta,\kappa}} \left|e^{zt_{n}}\right||z|^{-1}  |dz|
& = \int^{\infty}_{\frac{\pi}{\tau \sin \theta}} e^{r t_{n} \cos \theta} r ^{-1}  dr \leq c \tau^{k} t^{-k}_{n}
\end{split}
\end{equation}
with $1\leq (\frac{\sin \theta}{\pi})^{k} \tau^{k} r^{k}$, $r\tau\geq \frac{\pi}{\sin \theta}$.
 { Then we have $\left\|I_{1}  \right\|_{ {L^2(\Omega)}}+\left\|I_{2}  \right\|_{ {L^2(\Omega)}}\leq \left\|J_v \right\|_{ {L^2(\Omega)}}.$ }

%?????????????3.6
%$G(t):=\frac{t^{l}}{l!} g^{(l-m)}(0)$
% \begin{equation*}
%\left\| V(t_{n}) - V^{n}  \right\|\leq \left\{ \begin{array}
% {l@{\quad} l}
% \left( c\tau^{l+1}t_{n}^{\alpha-m-1} + c\tau^{k}t_{n}^{\alpha+l-k-m} \right) \left\| g^{(l-m)}(0)\right\|, & l=0,1,3,5,\cdots\\ \\
%   \left( c\tau^{l+2}t_{n}^{\alpha-m-2} + c\tau^{k}t_{n}^{\alpha+l-k-m} \right) \left\| g^{(l-m)}(0)\right\|,&l=2,4,6,\cdots.
% \end{array}
% \right.
%\end{equation*}

According to  Lemmas \ref{lemma3.9} and \ref{lemma3.10}, and the general source function  \eqref{gs3.3}, i.e.,
%\begin{equation*}
%\begin{split}
%\frac{t^{m-1}}{(m-1)!} \ast g(t)=G(t) & =  \sum_{l=0}^{k+m-1}  \frac{t^l}{l!}  G^{(l)}(0)  + \frac{t^{k+m-1}}{(k+m-1)!}  \ast  G^{(k+m)}(t)\\
%%& =  \sum_{j=-m}^{k-1}  \frac{t^{(j+m)}}{(j+m)!}  g^{(j)}(0)  + \frac{t^{k+m-1}}{(k+m-1)!}  \ast  g^{(k)}(t).
%& =  \sum_{l=0}^{k+m-1}  \frac{t^l}{l!} g^{(l-m)}(0)  + \frac{t^{k+m-1}}{(k+m-1)!}  \ast  g^{(k)}(t).
%\end{split}
%\end{equation*}
$$G(t) = \sum_{l=0}^{k-1}  \frac{t^{(l+m)}}{(l+m)!}  g^{(l)}(0)  + \frac{t^{k+m-1}}{(k+m-1)!}  \ast  g^{(k)}(t),$$
it yields $\left\|I_3 \right\|_{ {L^2(\Omega)}}\leq \left\|J_g \right\|_{ {L^2(\Omega)}}+ c\tau^{k}  \int_{0}^{t_{n}}  (t_n-s)^{\alpha-1} \left\| g^{(k)}(s) \right\|_{ {L^2(\Omega)}}  ds$.
% \begin{equation*}
%\left\|J_g \right\|\leq \left\{ \begin{array}
% {l@{\quad} l}
% c \sum\limits_{l=0}^{k-1}\left( \tau^{l+m+1}t_{n}^{\alpha-m-1} + \tau^{k}t_{n}^{\alpha+l-k} \right) \left\| g^{(l)}(0)\right\|, & l+m=1,3,5,\cdots\\
%  c \sum\limits_{l=0}^{k-1} \left( \tau^{l+m+2}t_{n}^{\alpha-m-2} + \tau^{k}t_{n}^{\alpha+l-k} \right) \left\| g^{(l)}(0)\right\|,&l+m=2,4,6,\cdots.
% \end{array}
% \right.
%\end{equation*}
The proof is completed.
\end{proof}

\section{Convergence analysis: Singular  source function $t^{\mu}q$, $\mu>-1$}\label{Se:WSST}
%Form Theorem \ref{addtheorema3.1} and Theorem \ref{addtheorema3.2}, it seems that there are no difference between  ID1-BDF2 and ID2-BDF2
%for general source function. However, both of them are very different for the  singular  source function with the form  $t^{\mu}q(x)$.
%
%\subsection{ID3-BDF3 for singular source term}
We next  consider the singular source term $g(t)=t^{\mu}q$ with $ \mu> -1$ for  \eqref{rrfee}.
We  introduce the polylogarithm function
\begin{equation}\label{polylogarithm function}
Li_{p}(\xi)= \sum_{j=1}^{\infty} \frac{\xi^{j}}{j^{p}},~~p\notin \mathbb{N}
\end{equation}
with the Riemann zeta function $\zeta(p)=Li_{p}(1)$.

Let  $G(t)=J^{m}g(t) = \frac{\Gamma(\mu+1)t^{\mu+m}}{\Gamma(\mu+m+1)}q$ with the  Laplace transform  $\widehat{G}(z) =  \frac{\Gamma(\mu+1)}{z^{\mu+m+1}} q$.
From  \eqref{LT} and \eqref{DLT},   it yields  the  continuous  solution
\begin{equation*}%\label{add4.1}
\begin{split}
V(t)
&= \frac{1}{2\pi i} \int_{\Gamma_{\theta, \kappa}} e^{zt}  (z^{\alpha}-A)^{-1}\left( z^{-1}Av   + \frac{\Gamma(\mu+1)}{z^{\mu+1}} q \right) dz,
\end{split}
\end{equation*}
and the  discrete solution
\begin{equation*}
V^{n}=\frac{\tau}{2\pi i}\int_{\Gamma^{\tau}_{\theta,\kappa}} e^{zt_{n}} (\delta^{\alpha}_{\tau, k}(e^{-z\tau})-A)^{-1} \delta^{m}_{\tau, k}(e^{-z\tau})  \left( \frac{\gamma_{m}(e^{-z\tau})}{m!} \tau^{m} A v + \widetilde{G}(e^{-z\tau}) \right) dz
\end{equation*}
%with $\gamma_{m}(\xi)=\sum^{\infty}_{n=1}n^{m} \xi^{n}$  and $\Gamma^{\tau}_{\theta, \kappa}=\{z\in \Gamma_{\theta, \kappa}: |\Im z|\leq \pi / \tau\}$.
with
\begin{equation*}
\begin{split}
\widetilde{G}(\xi)
= \sum^{\infty}_{n=1} G^{n} \xi^{n}
&= q \frac{ \Gamma(\mu+1) \tau^{\mu+m} \sum^{\infty}_{n=1} n^{\mu+m}\xi^{n}}{\Gamma(\mu+m+1)}
= q \frac{ \Gamma(\mu+1) \tau^{\mu+m} Li_{-\mu-m}(\xi)}{\Gamma(\mu+m+1)}.
\end{split}
\end{equation*}
%\begin{equation}
%V^{n}=\frac{\tau}{2\pi i}\int_{\Gamma^{\tau}_{\theta,\kappa}}\!\!\!\! e^{zt_n} \left( \delta^{\alpha}_{\tau, k}(e^{-z\tau}) -A\right)^{-1} \delta^{m}_{\tau, k}(e^{-z\tau}) \!\left[ \frac{\gamma_{m}(\xi)}{m!} \tau^{m} Av + \widetilde{G} (e^{-z\tau}) \right]dz
%\end{equation}
\begin{lemma}\cite{JLZ:2016}\label{addLemma:LipCA}
Let $|z\tau|\leq \frac{\pi}{\sin\theta}$ and  $\theta>\pi/2$ be close to $\pi/2$.
Then we have
\begin{equation*}%\label{LpSE}
Li_{p}(e^{-z\tau}) = \Gamma(1-p)(z\tau)^{p-1} + \sum_{j=0}^{\infty} (-1)^{j} \zeta(p-j)\frac{(z\tau)^{j}}{j!},~p\notin \mathbb{N},
\end{equation*}
and the  infinite  series converges absolutely. { Here $\zeta$ denotes the Riemann zeta function.}
\end{lemma}
\begin{lemma}\label{addLemma4.5}
Let  $\gamma_{\mu+m}(\xi)=\sum^{\infty}_{n=1}n^{\mu+m} \xi^{n}$
with $1\leq m\leq k\leq 6$. Then there { exists} a  positive constant $c$ such that
\begin{equation*}
\left| \frac{\gamma_{\mu+m}(e^{-z\tau})}{\Gamma{(\mu+m+1)}}  \tau^{\mu+m+1} - \frac{1}{z^{\mu+m+1}} \right| \leq
c \tau^{\mu+m+1}, ~~\mu>-1.
\end{equation*}
\end{lemma}
\begin{proof}
 From Lemma \ref{Lemma nn3.3}, the desired result is obtained with $\mu \in \mathbb{N}$. We next prove the case $\mu \notin \mathbb{N}$.
Using  \eqref{polylogarithm function} and Lemma \ref{addLemma:LipCA}, { it yields}
\begin{equation*}
\begin{split}
&\left| \frac{\gamma_{\mu+m}(e^{-z\tau})}{\Gamma{(\mu+m+1)}}  \tau^{\mu+m+1} - \frac{1}{z^{\mu+m+1}} \right| \\
&=    \left| \frac{\tau^{\mu+m+1}}{\Gamma(\mu+m+1)} \left(Li_{-\mu-m}(e^{-z\tau})  -  \frac{\Gamma(\mu+m+1)}{(z\tau)^{\mu+m+1}} \right)  \right|\\
&\leq  \frac{\tau^{\mu+m+1}}{\Gamma(\mu+m+1)} \left| \sum_{j=0}^{\infty} (-1)^{j} \zeta(-\mu-m-j)\frac{(z\tau)^{j}}{j!} \right|
\leq  c \tau^{\mu+m+1}.
\end{split}
\end{equation*}
 The proof is completed.
\end{proof}

\begin{lemma}\label{addLemma 4.3}
Let $\delta_{\tau, k}(\xi)$ with $k \leq 6$ be given in \eqref{2.2}. Then there { exists} a  positive constant $c$ such that
\begin{equation*}
\left\| \left( \delta^{\alpha}_{\tau, k}(e^{-z\tau}) -A\right)^{-1} \delta^{m}_{\tau, k}(e^{-z\tau}) - (z^{\alpha}-A)^{-1} z^{m} \right\| \leq c\tau^{k}|z|^{k+m-\alpha} \quad \forall z\in \Gamma^{\tau}_{\theta,\kappa}.
\end{equation*}
%where $\theta\in (\pi/2,\pi)$ is  sufficiently close to $\pi/2$.% and $\gamma_{2}(\xi)$ is given in Lemma \ref{Lemma 3.4}. %% \frac{e^{-z\tau}+e^{-2z\tau}}{2(1-e^{-z\tau})^3}
\end{lemma}
\begin{proof}
Let
\begin{equation*}
 \left( \delta^{\alpha}_{\tau, k}(e^{-z\tau}) -A\right)^{-1} \delta^{m}_{\tau, k}(e^{-z\tau})    - (z^{\alpha}-A)^{-1} z^{m} = I + II
\end{equation*}
with
\begin{equation*}
\begin{split}
I  & = \left( \delta^{\alpha}_{\tau, k}(e^{-z\tau}) -A\right)^{-1} \left( \delta^{m}_{\tau, k}(e^{-z\tau})   -  z^{m} \right),  \\
II & = \left( \left( \delta^{\alpha}_{\tau, k}(e^{-z\tau}) -A\right)^{-1} - (z^{\alpha}-A)^{-1} \right) z^{m}.
\end{split}
\end{equation*}
From  \eqref{discrete fractional resolvent estimate} and Lemma \ref{Lemma 2.3}, we obtain
\begin{equation*}
\|I\| \leq  c \tau^{k} |z|^{k+m-\alpha}.
\end{equation*}
Using  Lemma \ref{Lemma 2.3}, \eqref{discrete fractional resolvent estimate}, \eqref{fractional resolvent estimate} and the identity
$$\left( \delta^{\alpha}_{\tau, k}(e^{-z\tau}) -A\right)^{-1} - (z^{\alpha}-A)^{-1}=\left( z^{\alpha} - \delta^{\alpha}_{\tau, k}(e^{-z\tau}) \right) \left( \delta^{\alpha}_{\tau, k}(e^{-z\tau}) -A\right)^{-1} (z^{\alpha}-A)^{-1},$$
we estimate $II$ as following
\begin{equation*}
\|II\|  \leq c \tau^{k} |z|^{k +\alpha}  c |z|^{-\alpha} c |z|^{-\alpha}  |z|^{m} \leq c \tau^{k} |z|^{k+m-\alpha}.
\end{equation*}
According to the triangle inequality,  the desired result is obtained.
\end{proof}

\begin{theorem}\label{addtheorema4.1}
Let $V(t_{n})$ and $V^{n}$ be the solutions of  \eqref{rrfee} and \eqref{2.3}, respectively. Let $v\in L^{2}(\Omega)$ and $g(x,t)=t^{\mu}q $, $\mu > -1$, $q \in L^{2}(\Omega)$.  Then the following error estimate holds for any $t_n>0$
\begin{equation*}
\left\|V^{n}-V(t_{n})\right\|_{ {L^2(\Omega)}} \leq  \left\|J_v  \right\|_{ {L^2(\Omega)}}
+ c\tau^{\mu + m+1} t^{\alpha-m-1}_{n}\| q \|_{ {L^2(\Omega)}}  + c \tau^{k} t_{n}^{\alpha+\mu-k}\left\| q \right\|_{ {L^2(\Omega)}}
\end{equation*}
with $\left\|J_v  \right\|_{ {L^2(\Omega)}}$  { in \eqref{initialesti}}.
\end{theorem}
\begin{proof}
From Theorem \ref{addtheorema3.1}, the desired result is obtained with $\mu \in \mathbb{N}$. We next prove the case $\mu \notin \mathbb{N}$.
%Subtracting \eqref{add4.1} from \eqref{add4.4}, we obtain
Subtracting \eqref{LT} from \eqref{DLT}, we obtain
\begin{equation*}
V^{n}-V(t_{n})=I_{1}-I_{2} + I_{3} - I_{4},
\end{equation*}
where  $I_{1}$, $I_{2}$ { are defined by}  \eqref{ada3.a1}, and
\begin{equation*}
\begin{split}
%I_{1} =& \frac{1}{2\pi i} \int_{\Gamma^{\tau}_{\theta,\kappa}} e^{zt_{n}}  \mathcal{K}(z) v dz,~~\mathcal{K}(z){\rm ~in~ Lemma~\ref{addLemma 3.6}}, \\
%I_{2} =& \frac{1}{2\pi i} \int_{\Gamma_{\theta,\kappa}\backslash \Gamma^{\tau}_{\theta,\kappa}} e^{zt_{n}}  (z^{\alpha}-A)^{-1} z^{-1}  Av dz,\\
I_{3} =& \frac{1}{2\pi i}\int_{\Gamma^{\tau}_{\theta,\kappa}} e^{zt_{n}} \left[\left(\delta^{\alpha}_{\tau, k}(e^{-z\tau})-A \right)^{-1} \delta^{m}_{\tau, k}(e^{-z\tau}) \tau \widetilde{G} (e^{-z\tau}) - (z^{\alpha}-A)^{-1} z^{m} \widehat{G}(z)\right] dz, \\
I_{4} =& \frac{1}{2\pi i} \int_{\Gamma_{\theta,\kappa}\backslash \Gamma^{\tau}_{\theta,\kappa}} e^{zt_n}  (z^{\alpha}-A)^{-1} z^{m}  \widehat{G}(z) dz.
\end{split}
\end{equation*}

According to \eqref{ad3.I1} and \eqref{add3.162}, we estimate
{  $\left\|I_{1}  \right\|_{ {L^2(\Omega)}}+\left\|I_{2}  \right\|_{ {L^2(\Omega)}}\leq \left\|J_v \right\|_{ {L^2(\Omega)}}$. }
From \eqref{add3.16}, it leads to
\begin{equation*}
\begin{split}
\left\|I_{4}\right\|_{ {L^2(\Omega)}}
%&\leq c\int_{\Gamma_{\theta,\kappa}\backslash \Gamma^{\tau}_{\theta,\kappa}} {\left|e^{zt_{n}} \right||z|^{-\alpha}\left\|z^{m} \widehat{G}(z)\right\| } |dz| \\
&\leq c\int_{\Gamma_{\theta,\kappa}\backslash \Gamma^{\tau}_{\theta,\kappa}} {\left|e^{zt_{n}}\right||z|^{-\alpha} |z|^{-\mu-1} } \left\| q \right\|_{ {L^2(\Omega)}}  |dz|
\leq c\tau^{k} t^{\alpha+\mu-k}_{n}  \left\| q \right\|_{ {L^2(\Omega)}}.
\end{split}
\end{equation*}
Finally we consider $I_{3}= I_{31} + I_{32}$
with
\begin{equation*}
\begin{split}
I_{31} =& \frac{1}{2\pi i}\int_{\Gamma^{\tau}_{\theta,\kappa}} e^{zt_{n}} \left( \delta^{\alpha}_{\tau, k}(e^{-z\tau}) -A\right)^{-1} \delta^{m}_{\tau, k}(e^{-z\tau})  \left(\tau \widetilde{G} (e^{-z\tau})
-   \widehat{G}(z)\right) dz,\\
I_{32} =& \frac{1}{2\pi i} \int_{\Gamma^{\tau}_{\theta,\kappa}} e^{zt_n}  \left(\left( \delta^{\alpha}_{\tau, k}(e^{-z\tau}) -A\right)^{-1} \delta^{m}_{\tau, k}(e^{-z\tau})  - (z^{\alpha}-A)^{-1} z^{m} \right) \widehat{G}(z) dz.
\end{split}
\end{equation*}
According to \eqref{discrete fractional resolvent estimate}, Lemmas \ref{Lemma 2.3} and \ref{addLemma4.5}, there exists
\begin{equation*}
\left\|I_{31}\right\|_{ {L^2(\Omega)}}
\leq c \tau^{\mu+m+1}  \left\| q \right\|_{ {L^2(\Omega)}} \int_{\Gamma^{\tau}_{\theta,\kappa}} \left|e^{zt_{n}}\right| |z|^{m-\alpha} |dz|
\leq c \tau^{\mu+m+1}  t_{n}^{\alpha-m-1} \left\| q \right\|_{ {L^2(\Omega)}}.
\end{equation*}
From  Lemma \ref{addLemma 4.3} and $\widehat{G}(z) =  \frac{\Gamma(\mu+1)}{z^{\mu+m+1}} q$, it yields
\begin{equation*}
\begin{split}
\left\|I_{32}\right\|_{ {L^2(\Omega)}}
\leq &  c \tau^{k} \left\| q \right\|_{ {L^2(\Omega)}} \int_{\Gamma^{\tau}_{\theta,\kappa}} \left|e^{zt_{n}}\right| |z|^{k+m-\alpha} |z|^{-\mu-m-1}  |dz|
\leq    c \tau^{k}  t_{n}^{\alpha+\mu-k}\left\| q \right\|_{ {L^2(\Omega)}}.
\end{split}
\end{equation*}
The proof is completed.
\end{proof}

\section{Convergence analysis: Source function $t^{\mu}\circ f(t)$ with $\mu>-1$}
Based on the analysis of   Sections 3 and 4,
we next  provide   the detailed error estimates for { {the}} model  \eqref{fee}  with the singular/low regularity   source function   $t^{\mu}\circ f(t)$.
%The symbol $\circ$ can be either the convolution $*$  or the  product.
\subsection{Convergence analysis: Convolution source function $t^{\mu} \ast f(t)$, $\mu>-1$}
%In the subsection, we consider convolution source term $g(t) =   t^{\mu} \ast f(t)$ with $\mu >-1$ for subdiffusion \eqref{fee}.
Let $f(t) = \sum_{j=0}^{k-1}  \frac{t^{j}}{j!}f^{(j)}(0) + \frac{t^{k-1}}{(k-1)!} \ast f^{(k)}(t)$.
Then  we obtain
%\begin{equation*}
%g(t) =  t^{\mu} \ast f(t) =  \frac{t^{\mu+1}  f(0)}{\mu+1}   + \frac{t^{\mu+2}  f'(0)}{(\mu+2)(\mu+1)}
%+ \frac{t^{\mu+3}  f''(0)}{(\mu+3)(\mu+2)(\mu+1)} +  t^{\mu} \ast \frac{t^{2}}{2}\ast f''(t).
%\end{equation*}
\begin{equation*}
g(t) =  t^{\mu} \ast f(t) =   \sum_{j=0}^{k-1} \frac{ \Gamma(\mu+1)t^{\mu+j+1}}{\Gamma(\mu+j+2)}f^{(j)}(0)  +  t^{\mu} \ast \frac{t^{k-1}}{(k-1)!} \ast f^{(k)}(t).
\end{equation*}
Let  $G(t) = J^{m}g(t)=\frac{\Gamma(\mu+1)t^{\mu+m}}{\Gamma(\mu+m+1)} \ast f(t)$ with $G^{(j)}(0) =0, j=0, \ldots, m-1$.  It yields
%\begin{equation*}
%\begin{split}
%G(t)
%& =  \frac{t^{\mu+4}  f(0)}{(\mu+4)(\mu+3)(\mu+2)(\mu+1)}   + \frac{t^{\mu+5}  f'(0)}{(\mu+5)(\mu+4)(\mu+3)(\mu+2)(\mu+1)}  \\
%& \quad + \frac{t^{\mu+6}  f'(0)}{(\mu+6)(\mu+5)(\mu+4)(\mu+3)(\mu+2)(\mu+1)} + \frac{t^{5}}{120}  \ast \left( t^{\mu} \ast f'''(t) \right).
%\end{split}
%\end{equation*}
\begin{equation*}
\begin{split}
G(t)
& =  \sum_{j=0}^{k-1} \frac{ \Gamma(\mu+1)t^{\mu+j+m+1}}{\Gamma(\mu+j+m+2)}f^{(j)}(0)  +  \frac{t^{k+m-1}}{(k+m-1)!} \ast t^{\mu} \ast f^{(k)}(t).
\end{split}
\end{equation*}
%where we use
%\begin{equation*}
%\begin{split}
%t^{\mu+3} \ast \frac{t^{2}}{2}
%&=  \int_{0}^{t} (t-s)^{\mu+3}  \frac{s^{2}}{2} ds = \frac{\mu+3}{6}\int_{0}^{t} (t-s)^{\mu+2}  s^3 ds \\
%&= \frac{(\mu+3)(\mu+2)}{24}\int_{0}^{t} (t-s)^{\mu+1}  s^4 ds= \frac{(\mu+3)(\mu+2)(\mu+1)}{120}\int_{0}^{t} (t-s)^{\mu}  s^5 ds \\
%&= \frac{(\mu+3)(\mu+2)(\mu+1)}{120} t^{5} \ast t^{\mu}.
%\end{split}
%\end{equation*}

\begin{lemma}\label{addLemma4.6}
Let $V(t_{n})$ and $V^{n}$ be the solutions of \eqref{rrfee} and \eqref{2.3}, respectively.
Let $v=0$, $G(t):=  \frac{t^{k+m-1}}{(k+m-1)!} \ast \left( t^{\mu} \ast f^{(k)}(t)\right)$ with $\mu>-1$ and $\int_{0}^{t}  (t - s)^{\alpha-1}  s^{\mu} \ast \left\| f^{(k)}(s) \right\|_{ {L^2(\Omega)}}ds<\infty $. Then the following error estimate holds for any $t_n>0$
\begin{equation*}
\begin{split}
\left\|V(t_{n})-V^{n}\right\|_{ {L^2(\Omega)}}
&\leq c\tau^{k} \int_{0}^{t_{n}}  (t_{n} - s)^{\alpha-1}  s^{\mu} \ast \left\| f^{(k)}(s) \right\|_{ {L^2(\Omega)}}ds \\
&\leq c\tau^{k} \int_{0}^{t_{n}} (t_{n} - s)^{\alpha+\mu} \left\| f^{(k)}(s) \right\|_{ {L^2(\Omega)}}ds.
\end{split}
\end{equation*}
\end{lemma}
\begin{proof}
From  Lemma \ref{lemma3.10} and $g^{(k)}(t) =t^{\mu} \ast f^{(k)}(t)$, we obtain
\begin{equation*}
\begin{split}
\left\|V(t_{n})-V^{n}\right\|_{ {L^2(\Omega)}}
&\leq c\tau^{k} \int_{0}^{t_{n}} (t_n-s)^{\alpha-1} \left\|  s^{\mu} \ast  f^{(k)}(s)  \right\|_{ {L^2(\Omega)}}ds \\
& \leq c\tau^{k} \int_{0}^{t_{n}}  (t_{n} - s)^{\alpha-1}  s^{\mu} \ast \left\| f^{(k)}(s) \right\|_{ {L^2(\Omega)}}ds\\
&= c\tau^{k}\left(\left(t^{\alpha-1}\ast t^{\mu}\right) \ast \left\| f^{(k)}(t) \right\|_{ {L^2(\Omega)}}\right)_{t=t_n}\\
&  \leq c\tau^{k} \int_{0}^{t_{n}}  (t_{n} - s)^{\alpha+\mu} \left\| f^{(k)}(s) \right\|_{ {L^2(\Omega)}}ds .
\end{split}
\end{equation*}
%or
%\begin{equation*}
%\begin{split}
%\left\|V(t_{n})-V^{n}\right\|
%&\leq c\tau^{2} \int_{0}^{t_{n}} (t_n-s)^{\alpha-1} \left\|   \int_{0}^{s} (s-w)^{\mu} f''(w) dw \right\|ds \\
%&\leq c\tau^{2} \int_{0}^{t_{n}} (t_n-s)^{\alpha-1}   \int_{0}^{s}  (s-w)^{\mu} \left\|f''(w)\right\| dw ds \\
%& = c\tau^{2} \int_{0}^{t_{n}}  (t_{n} - s)^{\alpha-1}  s^{\mu} \ast \left\| f''(s) \right\|ds.
%\end{split}
%\end{equation*}
The proof is completed.
\end{proof}
\begin{theorem}\label{addtheorema4.2}
Let $V(t_{n})$ and $V^{n}$ be the solutions of \eqref{rrfee} and \eqref{2.3}, respectively.
Let $v\in L^{2}(\Omega)$, $g(t)=t^{\mu} \ast f(t) $,  $\mu>-1$ and  $f \in C^{k-1}([0,T]; L^{2}(\Omega))$,
$\int_{0}^{t}  (t - s)^{\alpha-1}  s^{\mu} \ast \left\| f^{(k)}(s) \right\|_{ {L^2(\Omega)}}ds<\infty $.  Then the following error estimate holds for any $t_n>0$
\begin{equation*}
\begin{split}
\left\|V^{n}-V(t_{n})\right\|_{ {L^2(\Omega)}}
&\leq \left\|J_v  \right\|_{ {L^2(\Omega)}}
 + c\tau^{k} \int_{0}^{t_{n}}  (t_{n} - s)^{\alpha-1}  s^{\mu} \ast \left\| f^{(k)}(s) \right\|_{ {L^2(\Omega)}}ds\\
 &\quad +c \sum_{j=0}^{k-1}\left(  \tau^{\mu + j + m+2} t^{\alpha-m-1}_{n}    +  \tau^{k} t_{n}^{\alpha+\mu+j-k+1}\right) \left\| f^{(j)}(0) \right\|_{ {L^2(\Omega)}}  \\
\end{split}
\end{equation*}
with $\left\|J_v  \right\|_{ {L^2(\Omega)}}$  { in \eqref{initialesti}}.
\end{theorem}

\begin{proof}
From  Theorem \ref{addtheorema4.1} and  Lemma \ref{addLemma4.6},   the desired result is obtained.
\end{proof}

\subsection{Convergence analysis: Product source function $t^{\mu}  f(t)$, $\mu>-1$}
Let $G(t) = J^{m}g(t)$ with $g(t)=t^{\mu}  f(t)$, $f(t) = \sum_{j=0}^{k-1}  \frac{t^{j}}{j!}f^{(j)}(0) + \frac{t^{k-1}}{(k-1)!} \ast f^{(k)}(t)$.
Then
\begin{equation*}
\begin{split}
G(t) %= \frac{t^{m-1}}{(m-1)!} \ast ( t^{\mu} f(t))
=  \sum_{j=0}^{k-1} \frac{ \Gamma(\mu+j+1)t^{\mu+j+m}}{\Gamma(\mu+j+m+1)j!}f^{(j)}(0)  + \frac{t^{m-1}}{\Gamma(m)} \ast h(t)
\end{split}
\end{equation*}
with $h(t)= t^{\mu} \left( \frac{t^{k-1}}{(k-1)!} \ast  f^{(k)}(t) \right)$.

\begin{lemma}\label{lem5.3ad}
Let $h(t)= t^{\mu} \left( \frac{t^{k-1}}{(k-1)!} \ast  f^{(k)}(t) \right) $ with $ \mu >-1$ and
$$f \in C^{k-1}([0,T]; L^{2}(\Omega)),~ \int_{0}^{t} \left\| f^{(k)}(s) \right\|_{ {L^2(\Omega)}} ds<\infty,~\int_{0}^{t} s^{\mu}  \left\| f^{(k)}(s) \right\|_{ {L^2(\Omega)}}ds < \infty.$$
 Then the following error estimate holds
 \begin{equation*}
\left\| h^{(k-1)}(0) \right\|_{ {L^2(\Omega)}} \leq  c\int_{0}^{t} s^{\mu}\left\| f^{(k)}(s) \right\|_{ {L^2(\Omega)}} ds,~~
h^{(l)}(0)=0~~\forall l\leq k-2,~~2\leq k\leq 6,
\end{equation*}
and
 \begin{equation*}
\left\| h^{(k-1)}(0) \right\|_{ {L^2(\Omega)}} \leq  c\int_{0}^{t} s^{\mu}\left\| f^{(k)}(s) \right\|_{ {L^2(\Omega)}} ds,~~k=1.
\end{equation*}

%\begin{equation*}
%h^{(l)}(0)=0~~ \forall l\leq k-2  ~~{\rm and}~~
%\left\| h^{(k-1)}(0) \right\|_{ {L^2(\Omega)}} \leq  c\int_{0}^{t} s^{\mu}\left\| f^{(k)}(s) \right\|_{ {L^2(\Omega)}} ds.
%\end{equation*}
\end{lemma}
\begin{proof}
Let us consider the case $\mu \notin \mathbb{N}$,
since the  result is trivial if $\mu \in \mathbb{N}$.
Using Leibnitz's formula for the $l$-th { {order}}  derivative of the  function $h(t)$, we get
%According to  definition of $h(t)$, for $l=0, \ldots, k-2$, we have
\begin{equation}\label{addeq5.5}
\begin{split}
h^{(l)}(t) = \sum_{j=0}^{l} \binom{l}{j} \frac{\Gamma(\mu+1)}{\Gamma(\mu+1-j)}  t^{\mu-j} \left(  \frac{t^{k-1-l+j}}{\Gamma(k-l+j)}  \ast  f^{(k)}(t) \right)
~~\forall l\leq k-1.
\end{split}
\end{equation}

{ Case 1: $2\leq k\leq  6$. From \eqref{addeq5.5}, we have}
\begin{equation*}
\begin{split}
\left\| h^{(l)}(t) \right\|_{ {L^2(\Omega)}}
&\leq c\sum_{j=0}^{l}  t^{\mu-j} \int_0^ts^{k-1-l+j}\left\| f^{(k)}(t-s) \right\|_{ {L^2(\Omega)}} ds\\
&{ \leq  c\sum_{j=0}^{l}  t^{\mu-j} \int_0^tt^{k-1-l+j}\left\| f^{(k)}(t-s) \right\|_{ {L^2(\Omega)}} ds} \leq c t^{\mu+k-1-l}~~\forall l\leq k-2,
%\left(  \frac{t^{k-1-l+j}}{\Gamma(k-l+j)}  \ast  f^{(k)}(t) \right)
 %\leq \left(\sum_{j=0}^{l} \frac{\Gamma(l+1)}{\Gamma(j+1)\Gamma(l-j+1)\Gamma(j+1)} \frac{ \Gamma(\mu+1) }{\left|\Gamma(\mu+1-j)\right|}\right) t^{\mu+k-1-l} \int_{0}^{t} s^{\mu} \left| f^{(k)}(s) \right| ds
%\leq c t^{\mu+k-1-l}.
\end{split}
\end{equation*}
and    $h^{(l)}(0)=0$, $\forall l\leq k-2$, $k\geq 2$.

We next estimate the bound of  $h^{(k-1)}(0)$.
From the above inequality, it implies $h^{(k-1)}(0)=0$ for $\mu>0$.
For $-1<\mu<0$, { using \eqref{addeq5.5}, it yields}
\begin{equation*}
\begin{split}
\left\| h^{(k-1)}(t) \right\|_{ {L^2(\Omega)}}
&\leq c\sum_{j=0}^{k-1}  t^{\mu-j} \int_0^tt^{j}\left\| f^{(k)}(t-s) \right\|_{ {L^2(\Omega)}} ds\\
&= c\sum_{j=0}^{k-1}   \int_0^tt^{\mu}\left\| f^{(k)}(t-s) \right\|_{ {L^2(\Omega)}} ds\leq c\int_{0}^{t} s^{\mu}\left\| f^{(k)}(s) \right\|_{ {L^2(\Omega)}} ds.
 %\leq \left(\sum_{j=0}^{l} \frac{\Gamma(l+1)}{\Gamma(j+1)\Gamma(l-j+1)\Gamma(j+1)} \frac{ \Gamma(\mu+1) }{\left|\Gamma(\mu+1-j)\right|}\right) t^{\mu+k-1-l} \int_{0}^{t} s^{\mu} \left| f^{(k)}(s) \right| ds
%\leq c t^{\mu+k-1-l}.
\end{split}
\end{equation*}

{ Case 2: $k=1$. From the above inequality, the desired result is obtained}.
%\begin{equation*}
%\begin{split}
%\left| h^{(k-1)}(t) \right|
%& \leq \left(\sum_{j=0}^{k-1} \frac{\Gamma(k)}{\Gamma(j+1)\Gamma(k-j)\Gamma(j+1)} \frac{ \Gamma(\mu+1) }{\left|\Gamma(\mu+1-j)\right|}\right) t^{\mu} \int_{0}^{t} \left| f^{(k)}(s) \right| ds \\
%& \leq \left(\sum_{j=0}^{k-1} \frac{\Gamma(k)}{\Gamma(j+1)\Gamma(k-j)\Gamma(j+1)} \frac{ \Gamma(\mu+1) }{\left|\Gamma(\mu+1-j)\right|}\right) \int_{0}^{t} s^{\mu}\left| f^{(k)}(s) \right| ds ~ -1 < \mu \leq 0.
%\end{split}
%\end{equation*}
%The proof is completed.
\end{proof}

%Moreover, there exists
%\begin{equation}
%\begin{split}
%h^{(l)}(t) = \sum_{j=0}^{l} \binom{l}{j} \frac{\Gamma(\mu+1)}{\Gamma(\mu+1-j)}  t^{\mu-j} \left(  \frac{t^{k-1-l+j}}{\Gamma(k-l+j)}  \ast  f^{(k)}(t) \right)
%~~\forall l\leq k-1.
%\end{split}
%\end{equation}
%\begin{equation}\label{add5.2}
%\begin{split}
%h^{(k)}(t) = \sum_{j=1}^{k} \frac{\Gamma(k+1)}{\Gamma(j)\Gamma(k+1-j)} \frac{\Gamma(\mu+1)}{\Gamma(\mu+1-j)}  t^{\mu-j} \left(  \frac{t^{j-1}}{(j-1)!}  \ast  f^{(k)}(t) \right)  +   t^{\mu}   f^{(k)}(t).
%\end{split}
%\end{equation}
%Thus one has
%\begin{equation}\label{eq5.5}
%\frac{t^{m-1}}{(m-1)!} \ast h(t)  = \frac{t^{k+m-1}}{(k+m-1)!} h^{(k-1)}(0) +  \frac{t^{k+m-1}}{(k+m-1)!} \ast h^{(k)}(t).
%\end{equation}

\begin{lemma}\label{Lemma5.7}
Let $V(t_{n})$ and $V^{n}$ be the solutions of \eqref{rrfee} and \eqref{2.3}, respectively.
Let $v=0$, $G(t)= \frac{t^{m-1}}{(m-1)!} \ast \left[ t^{\mu} \left( \frac{t^{k-1}}{(k-1)!} \ast  f^{(k)}(t) \right) \right]$, $\mu>-1$ and $f\in C^{k-1}([0,T]; L^{2}(\Omega))$,
$$ \int_{0}^{t} \left\| f^{(k)}(s) \right\|_{ {L^2(\Omega)}} ds<\infty,~
\int_{0}^{t}s^{\frac{\mu-1}{2}} \left\| f^{(k)}(s) \right\|_{ {L^2(\Omega)}} ds<\infty$$
and
$$   \int_{0}^{t} (t-s)^{\alpha-1} s^{\mu}  \left\| f^{(k)}(s) \right\|_{ {L^2(\Omega)}}ds<\infty.$$
Then the following error estimate holds for any $t_n>0$
\begin{equation*}
\begin{split}
&\left\|V(t_{n})-V^{n}\right\|_{ {L^2(\Omega)}}\\
& \leq c\tau^{k} \left(t_n^{\alpha + \mu-1} \int_{0}^{t_{n}} \left\| f^{(k)}(s) \right\|_{ {L^2(\Omega)}}  ds+ t_n^{\alpha + \frac{\mu-1}{2}} \int_{0}^{t_{n}} s^{\frac{\mu-1}{2}}  \left\| f^{(k)}(s) \right\|_{ {L^2(\Omega)}}  ds \right.\\
 &\left.\qquad\quad+ \int_{0}^{t_{n}} (t_n-s)^{\alpha-1} s^{\mu}  \left\| f^{(k)}(s) \right\|_{ {L^2(\Omega)}}ds \right),~~\mu>-1.
\end{split}
\end{equation*}
\end{lemma}
\begin{proof}
Let us consider the case $\mu \notin \mathbb{N}$,
since the  result is trivial if $\mu \in \mathbb{N}$.
Let $h(t)= t^{\mu} \left( \frac{t^{k-1}}{(k-1)!} \ast  f^{(k)}(t) \right)$.
From  Lemma \ref{lem5.3ad}, we have
\begin{equation*}
G(t)= \frac{t^{m-1}}{(m-1)!} \ast h(t)  = \frac{t^{k+m-1}}{(k+m-1)!} h^{(k-1)}(0) +  \frac{t^{k+m-1}}{(k+m-1)!} \ast h^{(k)}(t).
\end{equation*}
According to  Theorem \ref{addtheorema4.1} and Lemma \ref{lemma3.10}, it yields  %%= \frac{1}{\Gamma(1-\alpha)}  t^{-\alpha} \ast f''(t)
\begin{equation*}
\begin{split}
\left\|V(t_{n})\!-\!V^{n}\right\|_{ {L^2(\Omega)}}
& \leq c \tau^{k} \left(  t_{n}^{\alpha-1}\left\| h^{(k-1)}(0) \right\|_{ {L^2(\Omega)}}  \!+\! \int_{0}^{t_{n}} (t_n-s)^{\alpha-1} \left\|  h^{(k)}(s) \right\|_{ {L^2(\Omega)}}\!ds \right)
%& \leq c\tau^{k} \left( \sum_{j=1}^{k} I_{j} + I_{k+1} + I_{k+2} \right).
\end{split}
\end{equation*}
with
\begin{equation*}
\begin{split}
h^{(k)}(t) = \sum_{j=1}^{k} \binom{k}{j} \frac{\Gamma(\mu+1)}{\Gamma(\mu+1-j)}  t^{\mu-j} \left(  \frac{t^{j-1}}{\Gamma(j)}  \ast  f^{(k)}(t) \right)
+t^\mu f^{(k)}(t).
\end{split}
\end{equation*}

Since
\begin{equation*}
\begin{split}
& \sum_{j=1}^{k} \int_{0}^{t_{n}} (t_n-s)^{\alpha-1} \left\|  s^{\mu-j} \left( s^{j-1} \ast  f^{(k)}(s) \right) \right\|_{ {L^2(\Omega)}}ds\\
&\quad =  \sum_{j=1}^{k}\int_{0}^{t_{n}} (t_n-s)^{\alpha-1} s^{\frac{\mu-1}{2}} \left\|  \int_{0}^{s}   s^{\frac{\mu-1}{2}} \frac{(s-w)^{j-1}}{ s^{j-1}}  f^{(k)}(w) dw \right\|_{ {L^2(\Omega)}}ds \\
&\quad \leq  k \int_{0}^{t_{n}} (t_n-s)^{\alpha-1} s^{\frac{\mu-1}{2}} \int_{0}^{t_{n}}  w^{\frac{\mu-1}{2}} \left\| f^{(k)}(w) \right\|_{ {L^2(\Omega)}} dw ds \\
&\quad= kB(\alpha, (\mu+1)/2) t_n^{\alpha + \frac{\mu-1}{2}} \int_{0}^{t_{n}} w^{\frac{\mu-1}{2}} \left\| f^{(k)}(w) \right\|_{ {L^2(\Omega)}}  dw,~~  -1<\mu<0,
\end{split}
\end{equation*}
where we use
\begin{equation*}
\int_{0}^{t_{n}} (t_n-s)^{\alpha-1} s^{\frac{\mu-1}{2}} ds = t_n^{\alpha + \frac{\mu -1}{2}} \int_{0}^{1} (1-s)^{\alpha-1} s^{\frac{\mu-1}{2}} ds
 = B\left(\alpha, \frac{\mu+1}{2}\right) t_n^{\alpha + \frac{\mu -1}{2}}.
\end{equation*}
Similar, we can estimate
\begin{equation*}
\begin{split}
& \sum_{j=1}^{k} \int_{0}^{t_{n}} (t_n-s)^{\alpha-1} \left\|  s^{\mu-j} \left( s^{j-1} \ast  f^{(k)}(s) \right) \right\|_{ {L^2(\Omega)}}ds\\
%& =  \sum_{j=1}^{k}\int_{0}^{t_{n}} (t_n-s)^{\alpha-1} s^{\frac{\mu-1}{2}} \left\|  \int_{0}^{s}   s^{\frac{\mu-1}{2}} \frac{(s-w)^{j-1}}{ s^{j-1}}  f^{(k)}(w) dw \right\|ds \\
%& \leq  c \int_{0}^{t_{n}} (t_n-s)^{\alpha-1} s^{\frac{\mu-1}{2}} \int_{0}^{t_{n}}  w^{\frac{\mu-1}{2}} \left\| f^{(k)}(w) \right\| dw ds \\
& \quad \leq kB(\alpha, \mu) t_n^{\alpha + \mu-1} \int_{0}^{t_{n}} \left\| f^{(k)}(w) \right\|_{ {L^2(\Omega)}}  dw,  ~~\mu>0.
\end{split}
\end{equation*}
On the other hand, we have
\begin{equation*}
\begin{split}
%I_{j} & = \int_{0}^{t_{n}} (t_n-s)^{\alpha-1} \left\|  s^{\mu-j} \left( \frac{s^{j-1}}{(j-1)!} \ast  f^{(k)}(s) \right) \right\|ds, \quad j=1, \ldots, k\\
\int_{0}^{t_{n}} (t_n-s)^{\alpha-1} \left\|  s^{\mu}   f^{(k)}(s) \right\|_{ {L^2(\Omega)}}ds\leq  \int_{0}^{t_{n}} (t_{n}-s)^{\alpha-1}  s^{\mu} \left\| f^{(k)}(s) \right\|_{ {L^2(\Omega)}} ds,~~\mu>-1,
\end{split}
\end{equation*}
and
\begin{equation*}
\begin{split}
 t_{n}^{\alpha-1} \left\| h^{(k-1)}(0) \right\|_{ {L^2(\Omega)}}
 &\leq c t_{n}^{\alpha-1}   \int_{0}^{t_n} s^{\mu}\left\| f^{(k)}(s) \right\|_{ {L^2(\Omega)}} ds\\
 &\leq c \int_{0}^{t_{n}} (t_{n}-s)^{\alpha-1}  s^{\mu} \left\| f^{(k)}(s) \right\|_{ {L^2(\Omega)}} ds,~\mu>-1.
 \end{split}
\end{equation*}
By the triangle inequality, the desired result is obtained.
\end{proof}
\begin{theorem}\label{Theorem5.8}
Let $V(t_{n})$ and $V^{n}$ be the solutions of \eqref{rrfee} and \eqref{2.3}, respectively. Let $v\in L^{2}(\Omega)$, $g(t)= t^{\mu}  f(t)$, $\mu> -1$ and $f\in C^{k-1}([0,T]; L^{2}(\Omega))$,
$$ \int_{0}^{t} \left\| f^{(k)}(s) \right\|_{ {L^2(\Omega)}} ds<\infty,~
\int_{0}^{t}s^{\frac{\mu-1}{2}} \left\| f^{(k)}(s) \right\|_{ {L^2(\Omega)}} ds<\infty$$
and
$$   \int_{0}^{t} (t-s)^{\alpha-1} s^{\mu}  \left\| f^{(k)}(s) \right\|_{ {L^2(\Omega)}}ds<\infty.$$
Then the following error estimate holds for any $t_n>0$
\begin{equation*}
\begin{split}
&\left\|V^{n}-V(t_{n})\right\|_{ {L^2(\Omega)}}\\
& \leq  \left\|J_v  \right\|_{ {L^2(\Omega)}}
 +   \sum_{j=0}^{k-1} \left(  c\tau^{\mu + j + m + 1} t_{n}^{\alpha - m - 1} + c\tau^{k} t_{n}^{\alpha+ \mu + j -k} \right) \left\| f^{(j)}(0) \right\|_{ {L^2(\Omega)}}  \\
&\quad + c\tau^{k} \left[t_n^{\alpha + \mu-1} \!\!\int_{0}^{t_{n}} \left\| f^{(k)}(s) \right\|_{ {L^2(\Omega)}}  ds+ t_n^{\alpha + \frac{\mu-1}{2}}\!\! \int_{0}^{t_{n}} s^{\frac{\mu-1}{2}}  \left\| f^{(k)}(s) \right\|_{ {L^2(\Omega)}}  ds \right.\\
 &\left.\qquad \quad\quad+ \int_{0}^{t_{n}} (t_n-s)^{\alpha-1} s^{\mu}  \left\| f^{(k)}(s) \right\|_{ {L^2(\Omega)}}ds \right]
\end{split}
\end{equation*}
with $\left\|J_v  \right\|_{ {L^2(\Omega)}}$  { in \eqref{initialesti}}.
\end{theorem}
\begin{proof}
According to Theorem \ref{addtheorema4.1},  Lemma \ref{Lemma5.7}, and similar treatment of the initial data $v$ in Theorem  \ref{addtheorema3.1},   the desired result is obtained.
\end{proof}

\section{Numerical experiments}\label{Se:numer}
For the sake of brevity, we { mainly} employ ID$m$-BDF$6$ method in \eqref{2.3} for simulating the  model \eqref{fee},
since the similar numerical results can be performed for  ID$m$-BDF$k$ with $1\leq m\leq k<6$.
We discretize the space direction by the  spectral collocation method  with the Chebyshev-Gauss-Lobatto points \cite{STW:2011}.
The { discrete} $L^2$-norm { ($||\cdot||_{l_2}$)} is used to measure the numerical errors at the terminal time, e.g., $t=t_N=1$.
Since the analytic solution is unknown, the convergence rate of the numerical results are computed by
\begin{equation*}
  {\rm Convergence ~Rate}=\frac{\ln \left(||u^{N/2}-u^{N}||_{ {l_2}}/||u^{N}-u^{2N}||_{ {l_2}}\right)}{\ln 2}.
\end{equation*}
%with $u^N=V^N+v$ in  \eqref{2.3}.
All the numerical experiments are programmed in Julia 1.8.5.
{ One message is that multiple-precision floating-point computation is necessary, in order to reduce the round-off errors in evaluating.}

Let $T=1$ and $\Omega=(-1, 1)$. Consider the following { two} examples:
\begin{description}
  \item[(a)] $v(x)=\sin(x)\sqrt{1-x^2} $  and $g(x,t)= 0.$
 % \item[(b)] $v(x)=0 $  and $  g(x,t)= (e^t-1) e^x (1+\chi_{(0,1)}(x)).$
  \item[(b)] $v(x)=\sin(x)\sqrt{1-x^2}$  and $ g(x,t)= (1+t^{\mu}) \circ (e^t+1) e^x \left(1+\chi_{\left(0,1\right)}\left(x\right)\right).$
\end{description}

Here  $G(x,t)=J^mg(x,t)=\frac{t^{m-1}}{\Gamma(m)}\ast g(x,t)$ in \eqref{smoothing1}  are calculated by JacobiGL Algorithm \cite{ChD:15,Hesthaven:07},
which is generating the nodes and weights of Gauss-Lobatto integral with  the weighting function such as $(1-t)^\mu$ or $(1+t)^\mu$.

\begin{table}[!ht]
{\footnotesize \begin{center}
\caption{Case {\bf(a)}: convergent order of ID$m$-BDF$6$.}
{ \begin{tabular}{|l| l l l |l l l|} \hline
         \multirow{2}{*}{$m$}  &             & $\alpha=0.3$&             &             & $\alpha=0.7$&               \\ \cline{2-7}
                               &  ${ N=200}$    & ${ N=400}$     & ${ N=800}$     &  ${ N=200}$    & ${ N=400}$     & ${ N=800}$       \\ \hline
%         \multirow{2}{*}{0}    & 4.2127e-06  & 2.1059e-06  & 1.0528e-06  & 4.7081e-06  & 2.3533e-06  & 1.1764e-06   \\
%                               &             & 1.0002      & 1.0001      &             & 1.0004      & 1.0002       \\  \hline
         \multirow{2}{*}{1}    & 2.8649e-08  & 7.1623e-09  & 1.7905e-09  & 5.8503e-08  & 1.4626e-08  & 3.6565e-09    \\
                               &             & 1.9999      & 2.0000      &             & 1.9999      & 2.0000        \\  \hline
         \multirow{2}{*}{2}    & 3.4067e-13  & 2.0551e-14  & 1.2734e-15  & 1.0704e-12  & 6.3277e-14  & 3.9010e-15    \\
                               &             & 4.0511      & 4.0124      &             & 4.0804      & 4.0197        \\ \hline
         \multirow{2}{*}{3}    & 8.5708e-14  & 6.4330e-15  & 4.1808e-16  & 2.2754e-13  & 1.9161e-14  & 1.2708e-15    \\
                               &             & 3.7358      & 3.9436      &             & 3.5699      & 3.9143        \\ \hline
         \multirow{2}{*}{4}    & 2.9657e-14  & 4.4445e-16  & 6.8021e-18  & 1.3125e-13  & 1.9603e-15  & 2.9951e-17    \\
                               &             & 6.0602      & 6.0298      &             & 6.0651      & 6.0323        \\  \hline
         \multirow{2}{*}{5}    & 3.6694e-14  & 5.4991e-16  & 8.4162e-18  & 1.5877e-13  & 2.3712e-15  & 3.6228e-17    \\
                               &             & 6.0602      & 6.0298      &             & 6.0652      & 6.0323        \\ \hline
         \multirow{2}{*}{6}    & 4.3721e-14  & 6.5521e-16  & 1.0027e-17  & 1.8626e-13  & 2.7815e-15  & 4.2496e-17    \\
                               &             & 6.0602      & 6.0298      &             & 6.0652      & 6.0324        \\ \hline
\end{tabular}\label{table:11}}
\end{center}}
\end{table}

\begin{table}[!ht]
{\footnotesize \begin{center}
\caption{Case {\bf(b)} with   convolution: convergent order of ID$m$-BDF$6$.}
{ \begin{tabular}{|l| l l l |l l l|} \hline
         \multirow{2}{*}{$m$}  &  \multicolumn{3}{c|}{$\alpha=0.3, \mu=-0.4$}  &  \multicolumn{3}{c|}{$\alpha=0.7, \mu=0.3$}     \\ \cline{2-7}
                              &  ${ N=200}$    & ${ N=400}$     & ${ N=800}$     &  ${ N=200}$    & ${ N=400}$     & ${ N=800}$       \\ \hline
%         \multirow{2}{*}{0}    & 1.7848e-09    & 4.4664e-10    & 1.1174e-10    & 3.6369e-09    & 9.0924e-10    & 2.2731e-10     \\
%                               &               & 1.9986        & 1.9989        &               & 2.0000        & 2.0000         \\  \hline
         \multirow{2}{*}{1}    & 2.8489e-08    & 7.1310e-09    & 1.7848e-09    & 5.8505e-08    & 1.4626e-08    & 3.6565e-09     \\
                               &               & 1.9982        & 1.9982        &               & 2.0000        & 2.0000         \\  \hline
         \multirow{2}{*}{2}    & 5.2663e-12    & 4.2479e-13    & 3.4655e-14    & 1.3146e-12    & 8.0904e-14    & 5.0214e-15     \\
                               &               & 3.6319        & 3.6156        &               & 4.0222        & 4.0100         \\ \hline
         \multirow{2}{*}{3}    & 1.4691e-13    & 7.0242e-15    & 4.2499e-16    & 2.4212e-13    & 1.8753e-14    & 1.2629e-15     \\
                               &               & 4.3865        & 4.0468        &               & 3.6905        & 3.8923         \\ \hline
         \multirow{2}{*}{4}    & 1.2361e-13    & 1.8343e-15    & 2.7925e-17    & 2.6971e-13    & 4.0850e-15    & 6.2856e-17     \\
                               &               & 6.0743        & 6.0375        &               & 6.0449        & 6.0221         \\  \hline
         \multirow{2}{*}{5}    & 1.5665e-13    & 2.3252e-15    & 3.5406e-17    & 3.3626e-13    & 5.0941e-15    & 7.8391e-17     \\
                               &               & 6.0187        & 6.0093        &               & 6.0446        & 6.0220         \\ \hline
         \multirow{2}{*}{6}    & 1.8965e-13    & 2.8153e-15    & 4.2870e-17    & 4.0290e-13    & 6.1045e-15    & 9.3946e-17     \\
                               &               & 6.0739        & 6.0372        &               & 6.0444        & 6.0219         \\ \hline
\end{tabular}\label{table:33}}
\end{center}}
\end{table}

\begin{table}[!ht]
{\footnotesize \begin{center}
\caption{Case {\bf(b)} with   product: convergent order of ID$m$-BDF$6$.}
{ \begin{tabular}{|l| l l l |l l l|} \hline
         \multirow{2}{*}{$m$}  & \multicolumn{3}{c|}{$\alpha=0.3, \mu=-0.4$}  &  \multicolumn{3}{c|}{$\alpha=0.7, \mu=0.3$}    \\ \cline{2-7}
                              &  ${ N=200}$    & ${ N=400}$     & ${ N=800}$     &  ${ N=200}$    & ${ N=400}$     & ${ N=800}$       \\ \hline
         \multirow{2}{*}{1}    & 5.6978e-06    & 1.8615e-06    & 6.0971e-07   & 7.0952e-07    & 1.7474e-07    & 4.3146e-08     \\
                               &               & 1.6138        & 1.61033      &               & 2.0216        & 2.0179         \\  \hline
         \multirow{2}{*}{2}    & 4.7855e-09    & 7.9295e-10    & 1.3108e-10   & 6.1995e-11    & 5.8815e-12    & 5.7317e-13     \\
                               &               & 2.5933        & 2.59671      &               & 3.3979        & 3.35914        \\ \hline
         \multirow{2}{*}{3}    & 1.7095e-11    & 1.3407e-12    & 1.0870e-13   & 2.5149e-12    & 1.7869e-13    & 1.1382e-14     \\
                               &               & 3.6725        & 3.62449      &               & 3.8149        & 3.97267        \\ \hline
         \multirow{2}{*}{4}    & 9.9716e-13    & 1.4047e-14    & 1.7852e-16   & 7.2236e-13    & 1.0853e-14    & 1.6622e-16     \\
                               &               & 6.1494        & 6.2981       &               & 6.0565        & 6.02881        \\  \hline
         \multirow{2}{*}{5}    & 1.3010e-12    & 1.9509e-14    & 2.9866e-16   & 9.2805e-13    & 1.3945e-14    & 2.1371e-16     \\
                               &               & 6.0593        & 6.02948      &               & 6.0562        & 6.02803        \\ \hline
         \multirow{2}{*}{6}    & 1.5714e-12    & 2.3564e-14    & 3.6077e-16   & 1.1330e-12    & 1.7023e-14    & 2.6085e-16     \\
                               &               & 6.0592        & 6.02938      &               & 6.0565        & 6.02815        \\ \hline
\end{tabular}\label{table:44}}
\end{center}}
\end{table}

\begin{table}[!ht]
{\footnotesize \begin{center}
\caption{{ Case {\bf(b)} with   product: Comparison of Several Methods}.}
{ \begin{tabular}{|l| l l l |l l l|} \hline
\multirow{2}{*}{Scheme}           & \multicolumn{3}{c|}{$\alpha=0.3, \mu=-0.4$}  &  \multicolumn{3}{c|}{$\alpha=0.7, \mu=0.3$}    \\ \cline{2-7}
                               &  ${ N=200}$    & ${ N=400}$     & ${ N=800}$     &  ${ N=200}$    & ${ N=400}$     & ${ N=800}$       \\ \hline
\multirow{2}{*}{BDF2}          & 1.9996e-03    & 1.2957e-03    & 8.4321e-04   & 3.5185e-04    & 1.7321e-04    & 8.5275e-05     \\
                               &               & 0.6259        & 0.6198       &               & 1.0224        & 1.0223         \\  \hline
\multirow{2}{*}{Corr-BDF2}     & NaN           & NaN           & NaN          & 2.6235e-05    & 1.1807e-05    & 5.0842e-06     \\
                               &               &  ---             &  ---            &               & 1.1517        & 1.2156        \\ \hline
\multirow{2}{*}{ID2-BDF2}      & 4.7799e-05    & 1.1962e-05    & 2.9920e-06   & 4.6997e-05    & 1.1794e-05    & 2.9541e-06     \\
                               &               & 1.9985        & 1.9992       &               & 1.9945        & 1.9972         \\  \hline
\multirow{2}{*}{BDF4}          & 2.0020e-03    &1.2964e-03     &8.4336e-04    & 3.5870e-04    &1.7493e-04     &8.5705e-05   \\
                               &               & 0.6269        & 0.6202       &               & 1.0360        & 1.0292         \\  \hline
\multirow{2}{*}{Corr-BDF4}     & NaN           & NaN           & NaN          & NaN           & NaN           & NaN      \\
                               &               & ---         &  ---            &               &  ---             &  ---          \\ \hline
\multirow{2}{*}{ID4-BDF4}      & 1.8228e-09    & 1.1510e-10    & 7.2311e-12   & 2.7000e-09    & 1.6848e-10    & 1.0522e-11     \\
                               &               & 3.9850        & 3.9926       &               & 4.0023        & 4.0011         \\  \hline
\end{tabular}\label{table:adc}}
\end{center}}
\end{table}

Tables  \ref{table:11}   shows that
ID$m$-BDF$k$ with $k=6$ recovers    high-order convergence and this is in  agreement { with} Theorem \ref{addtheorema3.1}.
In fact, Table  \ref{table:11} indicates  an optimal error estimate of the Newton-Cotes rule $\mathcal{O}\left(\tau^{\min\{m+1,k\}}\right)$ for odd $m$ and $\mathcal{O}\left(\tau^{\min\{m+2,k\}}\right)$ for even $m$.

For the subdiffusion  model   \eqref{fee}, it  { may involve}    the low regularity or weakly singular  source terms \cite{Doerries:22,Maryshev:09,SBMB:03,SMB:09},  { e.g., }
$$g(x,t)=~t^{\mu}\ast f(x,t)~~{\rm or}~g(x,t)=t^{\mu} f(x,t), ~~\mu>-1.$$
In this case,   many time-stepping methods, including the correction of high-order  BDF schemes \cite{JLZ:2017,SC:2020}, are likely to exhibit a severe order reduction, { see Table \ref{table:adc}}, since it is required  the function $g\in C^{k-1}([0,T];L^{2}(\Omega))$.
{ In fact, for the low regularity  source term   $g(x,t)=t^{\mu}$, $\mu>0$,
the correction of  BDF$2$ (Corr-BDF$2$) scheme converges  with the order $\mathcal{O}(\tau^{1+\mu})$, see Lemma 3.2 in \cite{WZ:2020}.}
To fill in this gap, the desired $k$th-order convergence  rate can be recovered by ID$m$-BDF$k$ method,
 which are characterized by Theorem  \ref{addtheorema4.2}  and Theorem \ref{Theorem5.8}, see Tables  \ref{table:33} and \ref{table:44},  respectively.

\section{Conclusions}
The subdiffusion  models  { can involve}  the singular source term, which  exhibit  a severe order reduction by  many  time stepping methods.
In this work we first derive an  optimal error estimate  of the $k$th-order Newton-Cotes  rule $\mathcal{O}\left(\tau^{\min\{m+1,k\}}\right)$ for odd $m$ and $\mathcal{O}\left(\tau^{\min\{m+2,k\}}\right)$ for even $m$, $1\leq m\leq k\leq 6$,  under  the mild  regularity   of the source function.
Then  the desired $k$th-order convergence { rate} are well developed  by smoothing method  under the certainly singular source term{ s}.
It is interesting to design the numerical algorithms for the nonlinear fractional models.

\end{document}